\documentclass[12pt]{amsart}
\usepackage{amssymb,latexsym,tikz}
\usepackage[top=40mm, bottom=25mm, left=35mm, right=35mm]{geometry}
\usetikzlibrary{arrows,decorations.markings}
\tikzstyle arrowstyle=[scale=1]
\tikzstyle directed=[postaction={decorate,
decoration={markings,mark=at position .65 with {\arrow[arrowstyle]{stealth}}}}]

\usepackage{enumitem}
\usepackage{accents}
\usepackage{mathtools}
\usetikzlibrary{decorations.markings}
\tikzset{mid vert/.style={/utils/exec=\tikzset{every node/.append style={outer sep=0.8ex}},
postaction=decorate,decoration={markings,
mark=at position 0.5 with {\draw[-] (0,#1) -- (0,-#1);}}},
mid vert/.default=0.75ex}
\usepackage{amscd}
\usepackage{xypic}
\usepackage[utf8]{inputenc}

\begin{document}

\title 
[Triangulations and separated models]
{Triangulations of polygons and stacked simplicial complexes:
  separating their Stanley--Reisner ideals}

\author{Gunnar Fl{\o}ystad}
\address{Matematisk Institutt\\
         Postboks 7803\\
         5020 Bergen}
\email{gunnar@mi.uib.no}

\author{Milo Orlich}
\address{Aalto University\\
        Finland}
\email{milo.orlich@aalto.fi}


\date{\today}


\theoremstyle{plain}
\newtheorem{theorem}{Theorem}[section]
\newtheorem{corollary}[theorem]{Corollary}
\newtheorem*{main}{Main Theorem}
\newtheorem{lemma}[theorem]{Lemma}
\newtheorem{proposition}[theorem]{Proposition}
\newtheorem{conjecture}[theorem]{Conjecture}
\newtheorem{theoremp}{Theorem}

\theoremstyle{definition}
\newtheorem{definition}[theorem]{Definition}
\newtheorem{fact}[theorem]{Fact}
\newtheorem{obs}[theorem]{Observation}
\newtheorem{definisjon}[theorem]{Definisjon}
\newtheorem{problem}[theorem]{Problem}
\newtheorem{condition}[theorem]{Condition}

\theoremstyle{remark}
\newtheorem{notation}[theorem]{Notation}
\newtheorem{remark}[theorem]{Remark}
\newtheorem{example}[theorem]{Example}
\newtheorem{claim}{Claim}
\newtheorem{observation}[theorem]{Observation}


\newcommand{\psp}[1]{{{\bf P}^{#1}}}
\newcommand{\psr}[1]{{\bf P}(#1)}
\newcommand{\op}{{\mathcal O}}
\newcommand{\opw}{\op_{\psr{W}}}

\newcommand{\ini}[1]{\mathrm{in}(#1)}
\newcommand{\gin}[1]{\mathrm{gin}(#1)}
\newcommand{\kr}{{\Bbbk}}
\newcommand{\pd}{\partial}
\newcommand{\vardel}{\partial}
\renewcommand{\tt}{{\bf t}}


\newcommand{\coh}{{{\text{{\rm coh}}}}}


\newcommand{\modv}[1]{{#1}\text{-{mod}}}
\newcommand{\modstab}[1]{{#1}-\underline{\text{mod}}}

\newcommand{\sut}{{}^{\tau}}
\newcommand{\sumit}{{}^{-\tau}}
\newcommand{\til}{\thicksim}

\newcommand{\totp}{\mathrm{Tot}^{\prod}}
\newcommand{\dsum}{\bigoplus}
\newcommand{\dprod}{\prod}
\newcommand{\lsum}{\oplus}
\newcommand{\lprod}{\Pi}

\newcommand{\La}{{\Lambda}}

\newcommand{\sirstj}{\circledast}

\newcommand{\she}{\EuScript{S}\text{h}}
\newcommand{\cm}{\EuScript{CM}}
\newcommand{\cmd}{\EuScript{CM}^\dagger}
\newcommand{\cmri}{\EuScript{CM}^\circ}
\newcommand{\cler}{\EuScript{CL}}
\newcommand{\clerd}{\EuScript{CL}^\dagger}
\newcommand{\clerri}{\EuScript{CL}^\circ}
\newcommand{\gor}{\EuScript{G}}
\newcommand{\cF}{\mathcal{F}}
\newcommand{\cG}{\mathcal{G}}
\newcommand{\cM}{\mathcal{M}}
\newcommand{\cE}{\mathcal{E}}
\newcommand{\cI}{\mathcal{I}}
\newcommand{\cP}{\mathcal{P}}
\newcommand{\cK}{\mathcal{K}}
\newcommand{\cS}{\mathcal{S}}
\newcommand{\cC}{\mathcal{C}}
\newcommand{\cO}{\mathcal{O}}
\newcommand{\cJ}{\mathcal{J}}
\newcommand{\cU}{\mathcal{U}}
\newcommand{\cQ}{\mathcal{Q}}
\newcommand{\cX}{\mathcal{X}}
\newcommand{\cY}{\mathcal{Y}}
\newcommand{\cZ}{\mathcal{Z}}
\newcommand{\cV}{\mathcal{V}}

\newcommand{\mm}{\mathfrak{m}}

\newcommand{\dlim} {\varinjlim}
\newcommand{\ilim} {\varprojlim}

\newcommand{\CM}{\text{CM}}
\newcommand{\Mon}{\text{Mon}}


\newcommand{\Kom}{\text{Kom}}


\newcommand{\EH}{{\mathbf H}}
\newcommand{\res}{\text{res}}
\newcommand{\Hom}{\mathrm{Hom}}
\newcommand{\inhom}{{\underline{\mathrm{Hom}}}}
\newcommand{\Ext}{\mathrm{Ext}}
\newcommand{\Tor}{\mathrm{Tor}}
\newcommand{\ghom}{\mathcal{H}om}
\newcommand{\gext}{\mathcal{E}xt}
\newcommand{\id}{\text{{id}}}
\newcommand{\im}{\mathrm{im}\,}
\newcommand{\codim} {\mathrm{codim}\,}
\newcommand{\resol}{\text{resol}\,}
\newcommand{\rank}{\mathrm{rank}\,}
\newcommand{\lpd}{\mathrm{lpd}\,}
\newcommand{\coker}{\mathrm{coker}\,}
\newcommand{\supp}{\mathrm{supp}\,}
\newcommand{\Ad}{A_\cdot}
\newcommand{\Bd}{B_\cdot}
\newcommand{\Fd}{F_\cdot}
\newcommand{\Gd}{G_\cdot}


\newcommand{\sus}{\subseteq}
\newcommand{\sups}{\supseteq}
\newcommand{\pil}{\rightarrow}
\newcommand{\vpil}{\leftarrow}
\newcommand{\rpil}{\leftarrow}
\newcommand{\lpil}{\longrightarrow}
\newcommand{\inpil}{\hookrightarrow}
\newcommand{\pils}{\twoheadrightarrow}
\newcommand{\projpil}{\dashrightarrow}
\newcommand{\dotpil}{\dashrightarrow}
\newcommand{\adj}[2]{\overset{#1}{\underset{#2}{\rightleftarrows}}}
\newcommand{\mto}[1]{\stackrel{#1}\longrightarrow}
\newcommand{\vmto}[1]{\stackrel{#1}\longleftarrow}
\newcommand{\mtoelm}[1]{\stackrel{#1}\mapsto}
\newcommand{\bihom}[2]{\overset{#1}{\underset{#2}{\rightleftarrows}}}
\newcommand{\eqv}{\Leftrightarrow}
\newcommand{\impl}{\Rightarrow}

\newcommand{\iso}{\cong}
\newcommand{\te}{\otimes}
\newcommand{\into}[1]{\hookrightarrow{#1}}
\newcommand{\ekv}{\Leftrightarrow}
\newcommand{\equi}{\simeq}
\newcommand{\isopil}{\overset{\cong}{\lpil}}
\newcommand{\equipil}{\overset{\equi}{\lpil}}
\newcommand{\ispil}{\isopil}
\newcommand{\vvi}{\langle}
\newcommand{\hvi}{\rangle}
\newcommand{\susneq}{\subsetneq}
\newcommand{\sgn}{\text{sign}}


\newcommand{\xd}{\check{x}}
\newcommand{\ortog}{\bot}
\newcommand{\tL}{\tilde{L}}
\newcommand{\tM}{\tilde{M}}
\newcommand{\tH}{\tilde{H}}
\newcommand{\tvH}{\widetilde{H}}
\newcommand{\tvh}{\widetilde{h}}
\newcommand{\tV}{\tilde{V}}
\newcommand{\tS}{\tilde{S}}
\newcommand{\tT}{\tilde{T}}
\newcommand{\tR}{\tilde{R}}
\newcommand{\tf}{\tilde{f}}
\newcommand{\ts}{\tilde{s}}
\newcommand{\tp}{\tilde{p}}
\newcommand{\tr}{\tilde{r}}
\newcommand{\tfst}{\tilde{f}_*}
\newcommand{\empt}{\emptyset}
\newcommand{\bfa}{{\mathbf a}}
\newcommand{\bfb}{{\mathbf b}}
\newcommand{\bfd}{{\mathbf d}}
\newcommand{\bfl}{{\mathbf \ell}}
\newcommand{\bfx}{{\mathbf x}}
\newcommand{\bfm}{{\mathbf m}}
\newcommand{\bfv}{{\mathbf v}}
\newcommand{\bft}{{\mathbf t}}
\newcommand{\bbfa}{{\mathbf a}^\prime}
\newcommand{\la}{\lambda}
\newcommand{\bfen}{{\mathbf 1}}
\newcommand{\bfe}{{\mathbf 1}}
\newcommand{\ep}{\epsilon}
\newcommand{\en}{r}
\newcommand{\tu}{s}
\newcommand{\Sym}{\mathrm{Sym}}

\newcommand{\ome}{\omega_E}

\newcommand{\bevis}{{\bf Proof. }}
\newcommand{\demofin}{\qed \vskip 3.5mm}
\newcommand{\nyp}[1]{\noindent {\bf (#1)}}
\newcommand{\demo}{{\it Proof. }}
\newcommand{\demodone}{\demofin}
\newcommand{\parg}{{\vskip 2mm \addtocounter{theorem}{1}  
                   \noindent {\bf \thetheorem .} \hskip 1.5mm }}

\newcommand{\lcm}{{\mathrm{lcm}}}


\newcommand{\dl}{\Delta}
\newcommand{\cdel}{{C\Delta}}
\newcommand{\cdelp}{{C\Delta^{\prime}}}
\newcommand{\dlst}{\Delta^*}
\newcommand{\Sdl}{{\mathcal S}_{\dl}}
\newcommand{\lk}{\mathrm{lk}}
\newcommand{\lkd}{\lk_\Delta}
\newcommand{\lkp}[2]{\lk_{#1} {#2}}
\newcommand{\del}{\Delta}
\newcommand{\delr}{\Delta_{-R}}
\newcommand{\dd}{{\dim \del}}
\newcommand{\Del}{\Delta}

\renewcommand{\aa}{{\bf a}}
\newcommand{\bb}{{\bf b}}
\newcommand{\cc}{{\bf c}}
\newcommand{\xx}{{\bf x}}
\newcommand{\yy}{{\bf y}}
\newcommand{\zz}{{\bf z}}
\newcommand{\mv}{{\xx^{\aa_v}}}
\newcommand{\mF}{{\xx^{\aa_F}}}

\newcommand{\Symm}{\mathrm{Sym}}
\newcommand{\pnm}{{\bf P}^{n-1}}
\newcommand{\opnm}{{\go_{\pnm}}}
\newcommand{\ompnm}{\omega_{\pnm}}

\newcommand{\pn}{{\bf P}^n}
\newcommand{\hele}{{\mathbb Z}}
\newcommand{\nat}{{\mathbb N}}
\newcommand{\rasj}{{\mathbb Q}}
\newcommand{\bfone}{{\mathbf 1}}

\newcommand{\dt}{\bullet}
\newcommand{\disk}{\scriptscriptstyle{\bullet}}

\newcommand{\cxF}{F_\dt}
\newcommand{\pol}{f}

\newcommand{\Rn}{{\mathbb R}^n}
\newcommand{\An}{{\mathbb A}^n}
\newcommand{\frg}{\mathfrak{g}}
\newcommand{\PW}{{\mathbb P}(W)}

\newcommand{\pos}{{\mathcal Pos}}
\newcommand{\g}{{\gamma}}

\newcommand{\Vaa}{V_0}
\newcommand{\Bp}{B^\prime}
\newcommand{\Bpp}{B^{\prime \prime}}
\newcommand{\bbp}{\mathbf{b}^\prime}
\newcommand{\bbpp}{\mathbf{b}^{\prime \prime}}
\newcommand{\bp}{{b}^\prime}
\newcommand{\bpp}{{b}^{\prime \prime}}

\newcommand{\oLa}{\overline{\Lambda}}
\newcommand{\ov}[1]{\overline{#1}}
\newcommand{\ovv}[1]{\overline{\overline{#1}}}
\newcommand{\tm}{\tilde{m}}
\newcommand{\po}{\bullet}

\newcommand{\surj}[1]{\overset{#1}{\twoheadrightarrow}}
\newcommand{\Supp}{\mathrm{Supp}}

\def\CC{{\mathbb C}}
\def\GG{{\mathbb G}}
\def\ZZ{{\mathbb Z}}
\def\NN{{\mathbb N}}
\def\RR{{\mathbb R}}
\def\OO{{\mathbb O}}
\def\QQ{{\mathbb Q}}
\def\VV{{\mathbb V}}
\def\PP{{\mathbb P}}
\def\EE{{\mathbb E}}
\def\FF{{\mathbb F}}
\def\AA{{\mathbb A}}

\newcommand{\oR}{\overline{R}}
\newcommand{\bfu}{{\mathbf u}}
\newcommand{\nn}{{\mathbf n}}
\newcommand{\oa}{\overline{a}}
\newcommand{\cop}{\text{cop}}
\renewcommand{\op}{\text{op}\,}
\renewcommand{\mm}{{\mathbf m}}
\newcommand{\ngmi}{\text{neg}}
\newcommand{\up}{\text{up}}
\newcommand{\dw}{\text{down}}
\newcommand{\diw}[1]{\widehat{#1}}
\newcommand{\di}[1]{\diw{#1}}
\newcommand{\bo}{b}
\newcommand{\ub}{u}
\newcommand{\fs}{\infty}
\newcommand{\ifst}{\infty}
\newcommand{\mon}{{mon}}
\newcommand{\cl}{\text{cl}}
\newcommand{\intr}{\text{int}}
\newcommand{\ul}[1]{\underline{#1}}
\renewcommand{\ov}[1]{\overline{#1}}
\newcommand{\bipil}{\leftrightarrow}
\newcommand{\bfc}{{\mathbf c}}
\renewcommand{\mp}{m^\prime}
\newcommand{\np}{n^\prime}
\newcommand{\Mod}{\text{Mod }}
\newcommand{\Sh}{\text{Sh } }
\newcommand{\st}{\text{st}}
\newcommand{\hM}{\tilde{M}}
\newcommand{\hs}{\tilde{s}}
\newcommand{\ee}{\mathbf{e}}
\renewcommand{\dd}{\mathbf{d}}
\renewcommand{\en}{{\mathbf 1}}
\long\def\ignore#1{}
\newcommand{\lex}{{\text{lex}}}
\newcommand{\ordGL}{\succeq_{\lex}}
\newcommand{\ordG}{\succ_{\lex}}
\newcommand{\ordML}{\preceq_{\lex}}
\newcommand{\ordM}{\prec_{\lex}}
\newcommand{\tLa}{\tilde{\Lambda}}
\newcommand{\tGa}{\tilde{\Gamma}}
\newcommand{\STS}{\text{STS}}
\newcommand{\ii}{{\rotatebox[origin=c]{180}{\scalebox{0.7}{\rm{!}}}}}
\newcommand{\jj}{\mathbf{j}}
\renewcommand{\mod}{\text{ mod}\,}
\newcommand{\shmod}{\texttt{shmod}\,}
\newcommand{\hf}{\underline{f}}
\newcommand{\Glim}{\lim}
\newcommand{\Gcolim}{\colim}
\newcommand{\fm}{f^{\underline{m}}}
\newcommand{\fn}{f^{\underline{n}}}
\newcommand{\gn}{g_{\mathbf{|}n}}
\newcommand{\se}[1]{\overline{#1}}
\newcommand{\cB}{{\mathcal{B}}}
\newcommand{\fin}{{\text{fin}}}
\newcommand{\Poset}{{\text{\bf Poset}}}
\newcommand{\Set}{{\text{\bf Set}}}
\newcommand{\Homi}{\mathrm{Hom}^{L}}
\newcommand{\Homb}{\mathrm{Hom}_S}
\newcommand{\Homu}{\mathrm{Hom}^u}
\newcommand{\bfnu}{{\mathbf 0}}
\renewcommand{\bfen}{{\mathbf 1}}
\newcommand{\semip}{up semi-finite }
\newcommand{\semim}{down semi-finite }
\newcommand{\dif}[1]{\di{#1}_{fin}}
\newcommand{\llin}{\raisebox{1pt}{\scalebox{1}[0.6]{$\mid$}}}
\newcommand{\promap}{\mathrlap{{\hskip 2.8mm}{\llin}}{\lpil}}
\newcommand{\dual}{{\widehat{}}}
\newcommand{\Bool}{{\bf Bool }}
\newcommand{\pro}{{pro}}
\newcommand{\ovF}{\overline{F}}
\newcommand{\TO}{{\mathrm {to}}}
\newcommand{\dT}{\overset{\rightarrow}{T}}
\def\cl{\overline}

\newlength{\dhatheight}
\newcommand{\doublehat}[1]{%
    \settoheight{\dhatheight}{\ensuremath{\widehat{#1}}}%
    \addtolength{\dhatheight}{-0.35ex}%
    \widehat{\vphantom{\rule{1pt}{\dhatheight}}%
    \smash{\widehat{#1}}}}

\newcommand{\napo}{natural }
\newcommand{\ord}{\preceq}
\newcommand{\iC}{C}
\newcommand{\Spol}{k[x_{\iC}]}

\newcommand{\colim}{\mathrm{colim}}
\newcommand{\oPsi}{\overline{\Psi}}

\begin{abstract}
A triangulation of a polygon has an associated Stanley--Reisner
ideal. We obtain a full algebraic and combinatorial understanding
of these ideals and describe their separated models.

More generally we do this for stacked simplicial complexes,
in particular for stacked polytopes.
\end{abstract}

\maketitle

\section{Introduction}


Triangulations of polygons constitute a basic yet rich topic going into 
many directions. The most classical fact about these is perhaps that they
are counted by the Catalan numbers, \cite[Chap.23]{Gri}.
Their Stanley--Reisner ideals seem hitherto not
to have been systematically studied.
Here we get a full understanding of their algebraic
and combinatorial nature.
Considerably more generally, we do this for 
the Stanley--Reisner ideals of {\em stacked simplicial complexes}.

\begin{example}\label{eks:intro-hept1}
  Consider the triangulation of the heptagon in Figure~\ref{figone}.
This may be built up step by step from triangles, by successively
  attaching the triangles
  \[ 127,\ 257,\ 567,\ 245,\ 234. \]
  Each triangle after the first
  is attached to a single edge of some earlier triangle. This is
  a type of shelling called a {\em stacking},
\begin{figure}[htbp]
\begin{center}
  \begin{tikzpicture}[dot/.style={draw,fill,circle,inner sep=1pt},scale=1.1]
\draw [help lines,white] (-1.5,-1.5) grid (1.5,1.2);
\draw [fill=gray, opacity=0.3] (0,1)--(.79,.63)--(.98,-.23)--(.43,-.9)
  --(-.43,-.9)--(-.98,-.23)--(-.79,.63)--(0,1);
  \foreach \l [count=\n] in {0,1,2,3,4,5,6} {
    \pgfmathsetmacro\angle{90-360/7*(\n-1)}
      \node[dot] (n\n) at (\angle:1) {};
    \fill (\angle:1)  circle (0.08);
  }
  \foreach \l [count=\n] in {0,1,2,3,4,5,6} {
    \pgfmathsetmacro\angle{90-360/7*(\n-1)}
      \node (m\n) at (\angle:1.3) {$\n$};
  }
  \draw[thick] (n1) -- (n2) -- (n3) -- (n4) -- (n5) -- (n6) -- (n7) -- (n1);
  \draw[thick] (n2)--(n7)--(n5)--(n2)--(n4);
\end{tikzpicture}
\qquad\quad
\begin{tikzpicture}[dot/.style={draw,fill,circle,inner sep=1pt},scale=1.1]
\draw [help lines,white] (-1.5,-1.5) grid (5,1.2);
\draw [fill=gray, opacity=0.3] (0,1)--(.79,.63)--(.98,-.23)--(.43,-.9)
  --(-.43,-.9)--(-.98,-.23)--(-.79,.63)--(0,1);
\fill[red] (0,.77)  circle (0.07);
\fill[red] (-.75,-.17)  circle (0.07);
\fill[red] (.75,-.17)  circle (0.07);
\fill[red] (.25,-.4)  circle (0.07);
\fill[red] (-.18,.15)  circle (0.07);
\draw[thick,red] (.75,-.17)--(.25,-.4)--(-.18,.15)--(-.75,-.17);
\draw[thick,red] (-.18,.15)--(0,.77);
  \foreach \l [count=\n] in {0,1,2,3,4,5,6} {
    \pgfmathsetmacro\angle{90-360/7*(\n-1)}
      \node[dot] (n\n) at (\angle:1) {};
    \fill (\angle:1)  circle (0.08);
  }
  \draw[thick] (n1) -- (n2) -- (n3) -- (n4) -- (n5) -- (n6) -- (n7) -- (n1);
\draw[thick] (n2)--(n7)--(n5)--(n2)--(n4);
\node at (2,0) {$\rightsquigarrow$};
\fill[red] (3,.7)  circle (0.07);
\fill[red] (3,-.7)  circle (0.07);
\fill[red] (3.7,0)  circle (0.07);
\fill[red] (4.6,0)  circle (0.07);
\fill[red] (5.5,0)  circle (0.07);
\draw[thick,red] (3,.7)--(3.7,0)--(3,-.7);
\draw[thick,red] (3.7,0)--(5.5,0);
\end{tikzpicture}
\caption{}
\label{figone}
\end{center}
\end{figure}   and
  every triangulation of a polygon is a stacking.
  Moreover to a triangulation of the polygon we may associate a tree (drawn in red in Figure~\ref{figone}),
  showing how the triangles are attached to each other.
\end{example}

  This gives our two fundamental notions: That of stacking and
  the associated (hyper)tree.

  \medskip

  Let $X$ be a simplicial complex on a set $A$, i.e., a family of subsets
  of $A$ such that if $F \in X$ and $G \sus F$, then $G \in X$.
  Let $F_1, F_2, \dots, F_k$ be an ordering of the facets (the maximal faces)
  of $X$. We assume that the $F_i$'s all have the same cardinality. Let $X_p$
  be the simplicial complex generated by $F_1, \ldots, F_p$.

  The sequence $F_1,\dots,F_k$ is a {\em stacking} of $X$ if each $F_p$ is attached to
  $X_{p-1}$ along a single codimension-one face of $X_{p-1}$. So we may
  write $F_p = G_p \cup \{v_p\}$ where $G_p$ is a face of $X_{p-1}$ and
  $v_p$ is not a vertex of $X_{p-1}$.
  This is a shelling, but a particularly simple kind of shelling, since each
  $F_p$ is attached to a single codimension-one face, in contrast to
  a union of one or more such faces. A simplicial complex that has 
  a stacking as above is called a {\em stacked simplicial complex}. Such 
  simplicial complexes have appeared in the literature also as ``facet 
  constructible complexes''~(see~\cite{DS17,Goe}). Although they have been
  previously studied
  from the point of view of commutative algebra, to the best of our knowledge
  it was with a different persepective than in this paper. For instance, in~\cite{DS17} 
  the focus is  more on the homological invariants and Cohen--Macaulayness 
  of such complexes.

  To a stacked simplicial complex $X$ we associate a (hyper)tree as in
  the example above. Let $V$ be an index set for
  the facets of $X$. For a codimension-one face~$G$ of $X$ which is on
  at least two facets, let $e_G = \{ v \in V \, | \, F_v \supseteq G \}$.
  This gives a hypergraph on $V$ whose edges are the sets $e_G$. In fact this hypergraph
  is a hypertree $T$: it is connected,
  each pair of edges intersects in at most one
  vertex, and there are no cycles. The hypertree~$T$ is an ordinary
  tree, like in Figure~\ref{figone}, when each codimension-one face is on at most
  two facets. Then $X$ is a triangulated ball. In fact, $X$ may then be
  realized as a stacked polytope (see~\cite{Grum}), and every stacked polytope is of this
  kind. For the relationship between stacked simplicial
  complexes and stacked polytopes, we refer to Section~4.5 of~\cite{Goe}. 

  \medskip
  Given an (ordinary) tree $T$, let $V$ be the vertices of $T$, and $E$ the
  edge set of~$T$. Let $\iC \sus E \times V$ be the incidence relation
  consisting of pairs $(e,v)$ such that $v$ is a vertex on the edge
  $e$. 
  Let $\Spol$ be the polynomial ring in the variables
  $x_{e,v}$ for $(e,v) \in \iC$. 
  We associate a squarefree monomial ideal $I(T)$
  in the polynomial ring $\Spol$ as follows. Given
  a pair of vertices $v,w$ of $T$, there is a unique path between $v$ and $w$
  in the tree $T$: %
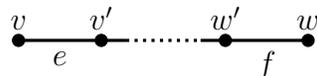
\[\begin{tikzpicture}[dot/.style={draw,fill,circle,inner sep=1pt},scale=1.1]
\fill (0,0)  circle (.08);
\fill (1,0)  circle (.08);
\fill (2.5,0)  circle (.08);
\fill (3.5,0)  circle (.08);
\draw[very thick] (0,0)--(1.3,0);
\draw[very thick,dotted] (1.3,0)--(2.2,0);
\draw[very thick] (2.2,0)--(3.5,0);
\coordinate [label=above: $v$] (v) at (0,0);
\coordinate [label=above: $v^\prime$] (vp) at (1,0);
\coordinate [label=below: $e$] (e) at (.5,0);
\coordinate [label=below: $f$] (f) at (3,0);
\coordinate [label=above: $w$] (w) at (3.5,0);
\coordinate [label=above: $w^\prime$] (wp) at (2.5,0);
\end{tikzpicture}\]
Associate to the pair of vertices $\{v,w\}$
the monomial $m_{v,w} = x_{e,v} \cdot x_{f,w}$.

The ideal $I(T)$ is the monomial ideal generated by the $m_{v,w}$ as $v$ and $w$
run through all distinct pairs of vertices of $V$.

We show that the Stanley--Reisner ring of any triangulation of a polygon is
obtained from $\Spol/I(T)$ by dividing out by a suitable regular sequence of
variable differences $x_{e,v^\prime} - x_{f,w^\prime}$.
More generally any Stanley--Reisner
ring of a {stacked simplicial complex} is obtained this way.
The rings $\Spol/I(T)$ for trees $T$ are thus the ``initial objects'' or
``free objects'' for Stanley--Reisner rings of stacked simplicial complexes.
Formulated otherwise, let $I$ be the Stanley--Reisner ring of a stacked
simplicial complex. The {\em separated models} of $I$ are one or more
of the~$I(T)$.

\begin{figure}
\begin{center}
\begin{tikzpicture}[dot/.style={draw,fill,circle,inner sep=1pt},scale=1.2]
\draw [help lines,white] (3,-1.8) grid (5,1.8);
\fill (3,.7)  circle (0.07);
\fill (3,-.7)  circle (0.07);
\fill (3.7,0)  circle (0.07);
\fill (4.6,0)  circle (0.07);
\fill (5.5,0)  circle (0.07);
\draw[thick] (3.7,0)--(3.052,.648);
\draw[thick] (3.7,0)--(3.052,-.648);
\draw[thick] (4.6,0)--(3.77,0);
\draw[thick] (5.5,0)--(4.67,0);
\node at (3.5,.5) {$a$};
\node at (3.5,-.5) {$b$};
\node at (4.2,.2) {$c$};
\node at (5.1,.24) {$d$};
\coordinate [label=left: $1$] (e) at (3,.7);
\coordinate [label=left: $3$] (3) at (3,-.7);
\coordinate [label=left: $2$] (2) at (3.65,0);
\coordinate [label=below: $4$] (4) at (4.6,0);
\coordinate [label=below: $5$] (5) at (5.5,0);
\end{tikzpicture}\qquad\qquad
\begin{tikzpicture}[dot/.style={draw,fill,circle,inner sep=1pt},scale=1.2]
\draw [help lines,white] (-1.5,-1.6) grid (1.5,1.4);
\draw [fill=gray, opacity=0.3] (0,1)--(.79,.63)--(.98,-.23)--(.43,-.9)
  --(-.43,-.9)--(-.98,-.23)--(-.79,.63)--(0,1);
  \foreach \l [count=\n] in {0,1,2,3,4,5,6} {
    \pgfmathsetmacro\angle{90-360/7*(\n-1)}
      \node[dot] (n\n) at (\angle:1) {};
    \fill (\angle:1)  circle (0.08);
  }
  \foreach \l [count=\n] in {0,1,2,3,4,5,6} {
    \pgfmathsetmacro\angle{90-360/7*(\n-1)}
      \node (p\n) at (\angle:1.35) {};
  }
  \node[font=\footnotesize] at (p1) {$(a,1)$};
  \node[font=\footnotesize] at (p2) {$(b,2)$};
  \node[font=\footnotesize] at (1.5,-.3) {$(d,5)$};
  \node[font=\footnotesize] at (p4) {$(c,4)$};
  \node[font=\footnotesize] at (p5) {$(a,2)$};
  \node[font=\footnotesize] at (-.73,-1.5) {$=(d,4)$};
  \node[font=\footnotesize] at (-1.5,-.3) {$(b,3)$};
  \node[font=\footnotesize] at (p7) {$(c,2)$};
  \draw[thick] (n1) -- (n2) -- (n3) -- (n4) -- (n5) -- (n6) -- (n7) -- (n1);
  \draw[thick] (n2)--(n7)--(n5)--(n2)--(n4);
\end{tikzpicture}
\caption{}
\label{figtwo}
\end{center}
\end{figure}

\begin{example}\label{eks:intro-hept2}
Consider the  directed tree in Figure~\ref{figtwo}.
The ideal $I(T)$ is generated by the ten monomials
  \begin{align*}
    m_{12} &= x_{a,1}x_{a,2}& m_{13} &= x_{a,1}x_{b,3}& 
    m_{14} &= x_{a,1}x_{c,4} &m_{15}& = x_{a,1}x_{d,5}\\
    m_{23} &=x_{b,2}x_{b,3}&  m_{24} &= x_{c,2}x_{c,4} & 
    m_{25} &= x_{c,2}x_{d,5}&  m_{34} &= x_{b,3}x_{c,4}\\
   &&  m_{35} &= x_{b,3}x_{d,5} &
    m_{45}&=x_{d,4}x_{d,5}.&&
    \end{align*}
  Then $I(T)$ is the Stanley--Reisner ideal of a stacked simplicial complex
  of dimension~$3$ (with facets of cardinality $4$) with eight vertices and
  five facets. Dividing out by the variable difference $x_{a,2} - x_{d,4}$,
  we get the Stanley--Reisner ring of the triangulation of the  heptagon
  $\Spol/(I(T) + (x_{a,2} - x_{d,4}))$. Figure~\ref{figtwo}, on the right,
  shows the triangulation
  with our new labelings of the vertices.
\end{example}

The ideals $I(T)$ are introduced in \cite{AFL} where
they are shown to be all possible polarizations of the square of
the graded maximal ideal $(x_e)^2_{e \in E}$ in $k[x_e]_{e \in E}$. 
If~$I_X$ is the Stanley--Reisner ring of a stacked simplicial complex
we therefore have processes:
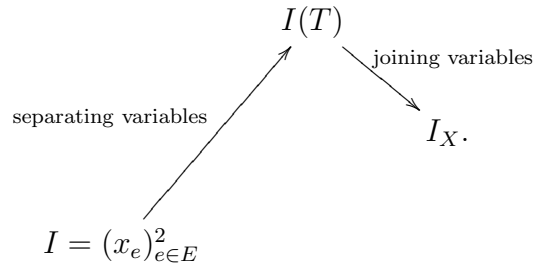
\[ \xymatrix{{}  & I(T) \ar^{\text{joining variables}}[rd] & \\
      & & I_X. \\
      I = (x_e)^2_{e \in E} \ar^{\text{separating variables}}[ruu] & &}
    \]
    Each of the arrows above preserves the graded Betti numbers. Hence every $I_X$
    has the same graded Betti numbers as a second power of a
    graded maximal ideal~$(x_e)_{e \in E}$.

\medskip
Let $\Spol_1 = \langle x_{e,v} \rangle_{(e,v) \in \iC}$ be the linear subspace
of one-dimensional forms in the polynomial ring $\Spol$. A subspace $L$ of this
linear space is a {\it regular linear space} if it has a basis consisting
of a regular sequence of variable differences for $\Spol/I(T)$.
The quotient ring by the space $L$ of linear forms
will still be a polynomial ring divided by a monomial ideal.
We show the following.

\medskip
\noindent {\bf Theorem \ref{thm:regtree-part}.}
{\em There is a one-to-one correspondence between regular linear spaces
for $\Spol/I(T)$ and partitions of the vertex set\/ $V$.
}

\medskip
In particular for the partition with one part, the whole of $V$, the
regular sequence consists of all variable differences $x_{e,v} - x_{e,w}$ for
$e = \{v,w\} \in E$, and the quotient ring is $k[x_e]_{e \in E}/m^2$, where
$m = ( x_e )_{e \in E}$ is the irrelevant maximal ideal.

\medskip
\noindent {\bf Theorem \ref{thm:regsq-partind}.}
{\em There is a one-to-one correspondence between regular linear spaces
for $\Spol/I(T)$ giving {\em squarefree} quotient rings, and partitions of
the vertex set\/ $V$ into sets of\/ {\em independent} vertices.
}
\medskip

These quotient
rings give the Stanley--Reisner rings of stacked simplicial complexes.
\medskip

In \cite{FlPar} the first author gives a one-to-one correspondence
between partitions of the vertex set
of a tree~$T$ into $(r+1)$ independent sets, and partitions of the edge set
of $T$ into $r$ sets. We recall this in the appendix. 
The above may then be reformulated as:

\medskip
\noindent {\bf Theorem \ref{thm:regsq-partedge}.}
{\em There is a one-to-one correspondence between regular linear spaces
for $\Spol/I(T)$ giving squarefree quotient rings, and partitions of
the edge set $E(T)$. Moreover, the dimension of the simplicial complex
 associated to this quotient Stanley--Reisner ring is one less than the number 
 of parts in the partition.}

\begin{example}\label{eks:intro-hept3}
Consider Example \ref{eks:intro-hept2} above. The regular
  linear space $L$ is the space $L = \langle x_{a,2} - x_{d,4} \rangle$.
  It corresponds to the partitions of vertices and partitions of edges
  of the tree in Figure~\ref{figthree}. 
  These partitions are respectively
  \[ V = \{1,5 \} \cup \{ 2\} \cup \{3\} \cup \{4\}, \qquad
    E = \{ a,d \} \cup \{b \} \cup \{c\}. \]
  There are three parts in the edge partition and so the dimension of
  the associated simplicial complex is one less, the dimension of the
  triangulated polygon.
\end{example}
\begin{figure}
\begin{center}
\begin{tikzpicture}[dot/.style={draw,fill,circle,inner sep=1pt},scale=1.4]
\draw[very thick,red] (3.7,0)--(3,.7);
\draw[very thick,green] (3.7,0)--(3,-.7);
\draw[very thick,blue] (4.6,0)--(3.7,0);
\draw[very thick,red] (5.5,0)--(4.6,0);
\fill (3,.7)  circle (0.07);
\fill (3,-.7)  circle (0.07);
\fill (3.7,0)  circle (0.07);
\fill (4.6,0)  circle (0.07);
\fill (5.5,0)  circle (0.07);
\node at (3.44,.44) {$a$};
\node at (3.2,-.3) {$b$};
\node at (4.2,.17) {$c$};
\node at (5.1,.206) {$d$};
\node at (2.8,.9) {$1$};
\node at (3.75,-.22) {$2$};
\node at (2.8,-.9) {$3$};
\node at (4.6,-.22) {$4$};
\node at (5.5,-.22) {$5$};
\end{tikzpicture}
\caption{}
\label{figthree}
\end{center}
\end{figure}

Finally we show the following.

\medskip
\noindent {\bf Theorem \ref{thm:ball-partedgeind}.}
{\em There is a one-to-one correspondence between regular linear spaces
for $\Spol/I(T)$ giving squarefree quotient rings whose associated simplicial
complex is a {\it triangulated ball}, and partitions of
the edge set $E(T)$ into sets of independent edges.
}
\medskip

In particular the last two theorems above give that
triangulations of {\it simplicial polygons} correspond to partitions of the
edges of trees $T$ into three parts, each part being a set of independent edges.
In particular only trees $T$ whose maximal vertex degree is $3$ arise
in this context (which is easy to see directly like in Example~\ref{eks:intro-hept1}).

\medskip
The organization of this article is as follows.
In Section~\ref{sec:sepjoin} we recall the notions of separating and joining
variables in monomial ideals. We
develop basic auxiliary results for doing this. We also recall
the notion of separated model.
In Sections~\ref{sec:triangulation} and~\ref{sec:tree}
we recall basic notions for simplicial
complexes. We define stacked simplicial complexes and hypertrees. We show
that the separated models of stacked simplicial complexes are
the ideals $I(T)$.

Section~\ref{sec:regtree} is the main technical part and gives the
combinatorial description of which
linear spaces of variable differences are regular linear spaces
for~$\Spol/I(T)$.
Section~\ref{sec:regsqfree} describes the regular linear spaces that
give squarefree quotient rings. Section~\ref{sec:part2} describes
the ordering relation between partitions of vertices that corresponds
to inclusion of regular linear spaces. Lastly in Section~\ref{sec:ball}
we describe those
regular linear spaces where the quotient ring is associated to
a triangulation of a ball, or equivalently of a stacked polytope. We also
describe the Stanley--Reisner ring of the boundary of these polytopes, which
are simplicial spheres.

The appendix 
recalls the correspondence between
partitions of vertices of $V(T)$ into independent sets and partitions of
edges of $E(T)$.

\subsection*{Acknowledgements}

We thank
Lars H\"allstr\"om, Veronica Crispin Quinonez, Russ Woodroofe
and the anonymous referee for all their comments, which  improved this paper.
In particular Lars H\"allstr\"om suggested the conceptual gain of
indexing the variables in $k[x_{\iC}]$ by pairs of edges and vertices $(e,v)$
such that $v \in e$, instead of letting the variables be indexed by
$E \times \{0,1\}$, as we did in a preliminary version of this article.
The second author was supported by the Finnish Academy of Science and Letters, with the \textit{Vilho, Yrj\"o and Kalle V\"ais\"al\"a Fund}. 

\medskip\noindent
Data availability statement: this manuscript has no associated data.

\section{Separations and joins for Stanley--Reisner ideals}
\label{sec:sepjoin}

We recall the notion of separation for monomial ideals $I$. It is
a converse to the notion of dividing a quotient ring $S/J$ out
by a variable difference which is a non-zero divisor.
When $I = I_X$ or $J = J_X$ is a Stanley--Reisner ideal, we descibe how the
simplicial complex $X$ transforms under these processes.

\medskip
For a set $V$ denote by $k[x_V]$ the polynomal ring in the variables
$x_v$ for $v \in V$. With some abuse of notation, for $R \sus V$
let $x_R$ denote the monomial
$\prod_{r \in R} x_r$. (This should not cause confusion since in the polynomial
ring we always surround~$x_V$ with square brackets.)

\subsection{Separating a variable}

The following definition is from \cite[Section 2]{FGH}.

\begin{definition}
Let $V^\prime \mto{p} V$ be a surjection of finite sets with the
cardinality of $V^\prime$ one more than that of $V$. Let $v_1$ and $v_2$
be the two distinct elements of $V^\prime$ which map to a single
element $v$ in $V$. Let $I$ be a monomial ideal in the polynomial ring
$k[x_V]$ and $J$ a monomial ideal in $k[x_{V^\prime}]$. We say
$J$ is a {\it simple separation} of $I$ if the following hold:
\begin{itemize}
\item[i.] The monomial ideal $I$ is the image of $J$ by the map
$k[x_{V^\prime}] \pil k[x_V]$.
\item[ii.] Both the variables $x_{v_1}$ and $x_{v_2}$ occur in some minimal
generators of $J$ (usually in distinct generators).
\item[iii.] The variable difference $x_{v_1} - x_{v_2}$ is a non-zero divisor
in the quotient ring $k[x_{V^\prime}]/J$.
\end{itemize}
More generally, if $V^\prime \mto{p} V$ is a surjection of finite sets
and $I \sus k[x_V]$ and $J \sus k[x_{V^\prime}]$ are monomial ideals such
that $J$ is obtained by a succession of simple separations of $I$,
$J$ is a {\it separation} of $I$.
 If $J$ has no further separation, we call $J$ a {\it separated model}
(of $I$).
\end{definition}

Let $X$ be a simplicial complex on the set $V$. This is a family of subsets of
$V$ such that $F \in X$ and $G \sus F$ implies $G \sus F$. 
The set of $v \in V$ with $\{v\} \in X$ is the {\it support} of $X$.
For $R \sus V$ the {\it restriction $X_R$} if the simplicial complex on
$R$ consisting of all $F \in X$ such
that $F \sus R$. Denote by $X_{-R}$ the restriction $X_{R^c}$ where
$R^c$ is the complement of $R$ in~$V$. The {\it link $\lk_X R$} is the
simplicial complex on $R^c$ consisting
of all $F \sus R^c$ such that $F \cup R \in X$. If $Y \sus X$ are simplicial
complexes, denote by $X \setminus Y$ the relative simplicial complex,
consisting of those $F \in X$ which are not in $Y$.

Let $I_X$ be the Stanley--Reisner ideal of $X$, the monomial ideal in
$k[x_V]$ whose generators are the monomials $x_R$ for $R \not \in X$. Suppose we use
$v$ to separate~$I_X$ to an ideal~$I_{X^\prime}$ in the polynomial
ring $k[x_{V^\prime}]$. Write the minimal set of monomial
generators of~$I_X$ as~$\cM_0 \cup \cM_v$,
where $\cM_0$ consists of those that do not contain $x_v$ and $\cM_v$ of those
of the form $x_v\cdot x_R$. The separated ideal $I_{X^\prime}$ will then
have minimal generators $\cM_0 \cup \cM_{v,1} \cup \cM_{v,2}$ (sets of monomials
in $k[x_{V^\prime}]$). Here $\cM_{v,i}$ consists of those minimal generators that
contain $x_{v_i}$. There is a bijection between $\cM_{v,1} \cup
\cM_{v,2}$ and $\cM_v$ by sending $x_{v_i}\cdot x_R$ to $x_v \cdot x_R$.

\subsection {Criteria for separating a variable}
\label{subsec:sepjoin-sep}

Here is a general description of how $I_X$ can be separated using
the variable $x_v$.

\begin{proposition}  \label{pro:sepjoin-critsep}
We may separate $I_X$ using the variable $x_v$ iff the following holds:
there is a partition of the faces into two non-empty parts
\[ X_{-\{v\}} \setminus \lk_X v = \cF_1 \sqcup \cF_2 \]
where each $\cF_i$ is closed under taking smaller sets in
the sense that if\/ $G \sus F$ and\/ $F \in \cF_i$, then either $G \in \cF_i$
or $G \in \lk_X v$.
The facets of the simplicial complex\/ $X^\prime$ in the separated
ideal\/ $I_{X^\prime}$ are then
obtained from the facets $F$ of\/ $X$ as follows:

\begin{itemize}
\item If\/ $F = G \cup \{v\}$ then $G \cup \{v_1,v_2\}$ is a facet of\/ $X^\prime$.
\item If\/ $F$ is in $\cF_1$ then $F \cup \{v_1\}$ is a facet of\/ $X^\prime$.
\item If\/ $F$ is in $\cF_2$ then $F \cup\{v_2\}$ is a facet of\/ $X^\prime$.
\end{itemize}
\end{proposition}

\begin{proof} Assume first that $I_{X^\prime}$ is a separation of $I_X$. 
If $F \in X_{-\{v\}} \setminus \lk_X v$ then $x_v x_F \in I_X$
(otherwise $F$ would be in $\lk_X v$).
Let $\cF_1$ be the set of those $F$ such that $x_{v_1}x_F$ is in
$I_{X^\prime}$ and $x_F$ is not in $I_{X^\prime}$, or equivalently
in $I_X$. Similarly define $\cF_2$ as the set of $F$ 
such that $x_{v_2}x_F$ is in $I_{X^\prime}$ and $x_F$ is not in $I_{X^\prime}$.
Let us show that $\cF_1$ and $\cF_2$ are disjoint.
If both $x_{v_1}x_F$ and $x_{v_2}x_F$ are in $I_{X^\prime}$, then
since $x_{v_1} - x_{v_2}$ is a nonzero divisor for
$k[x_{V^\prime}]/I_{X^\prime}$, we have $x_F$ in $I_{X^\prime}$ and so
in $I_X$. But then $F$ could not have been in $X$.

Suppose conversely we have the partition $\cF_1 \sqcup \cF_2$.
Let the ideal $I_{X^\prime}$ be constructed as just before this 
Subsection~\ref{subsec:sepjoin-sep} so it is generated by 
$\cM_0 \cup \cM_{v,1} \cup\cM_{v,2}$.
Let us show that $x_{v_1} - x_{v_2}$ is a nonzero divisor of
$k[x_{V^\prime}]/I_{X^\prime}$. Suppose $(x_{v_1} - x_{v_2})x_F$ is in $I_{X^\prime}$.
Then $x_{v_1} x_F$ and $x_{v_2} x_F$ are both in $I_{X^\prime}$.
We must show that $x_F$ is in $I_{X^\prime}$, or equivalently in $I_X$.
Suppose not, so $F$ is in $X_{-\{v\}}$. It is not in $\lk_X v$ since
$x_v x_F \notin I_X$. If say $F \in \cF_1$, then $F \cup \{v_1 \}$ is a
face of $X^\prime$ and $x_{v_1} x_F$ is not in $I_{X^\prime}$, a
contradiction. Hence $x_F$ is in $I_{X^\prime}$.  
\end{proof}

For $v \in V$ and $X$ a simplicial complex, let the neigborhood of $v$ be
\[ N(v) = \{w \mid \{ v,w \} \in X \} \sus V. \]
Note that $N(v)$ is non-empty iff $v$ is in the support of $X$,
in which case $v \in N(v)$. 

\begin{corollary} \label{cor:sep-flag}
  Let $X$ be a flag simplicial complex on $V$, i.e.,
  $I_X$ is generated by quadratic monomials. Let $v$ be in the support of $X$.
Suppose $X_{-N(v)} = X_1 \cup X_2$ where
    $X_1$ and $X_2$ are simplicial complexes
    supported on disjoint vertex sets $V_1$ and~$V_2$.
    Then  using $x_v$ the ideal
    $I_X \sus k[x_V]$ may be separated to an ideal $I_{X^\prime}
        \sus k[x_{V^\prime}]$. The facets of $X^\prime$ correspond to the
	facets $F$ of $X$ as follows:
        \begin{itemize}
        \item If\/ $F = G \cup \{v\}$ contains
          $v$ then $G \cup \{v_1,v_2\}$ is a facet of\/ $X^\prime$.
        \item If\/ $F$ is supported on $V_1$ then $F \cup \{v_1\}$ is
          a facet of\/ $X^\prime$.
        \item If\/ $F$ is supported on $V_2$ then $F \cup \{v_2\}$ is
          a facet of\/ $X^\prime$.
        \end{itemize}
      \end{corollary}

\begin{proof} Let $U = N(v) \setminus \{v\}$.
The link $\lk_X v$ is supported on $U \sus V$. 
Let $F \in X_{-\{v\}}\setminus \lk_X v$. Write $F$ as a disjoint union
$F_0 \cup F_1$ where $F_1 = F \cap U$.
Since $X$ is flag, if $F_0 = \emptyset$ we would have $F$ in $\lk_X v$.
So $F_0$ is non-empty and we show it is a subset of either $V_1$ or $V_2$. 
Let $a,b \in F_0$. Then $a$ is not in the link $\lk_X v$, 
so $a$ is either in $V_1$ or $V_2$. Similarly with $b$. 
So $\{a,b\}$ is in $X_{-N(v)}$ and so in either $X_1$ or $X_2$.
Hence $a,b$ are in the same set $V_i$. The upshot is that $F_0$
is a subset of either $V_1$ or $V_2$.

We then let $\cF_1$ be the set of those $F$ such that $F_0$ is a subset of $V_1$ and
similarly for $\cF_2$. These will then be disjoint and closed under taking
smaller sets.
\end{proof}

\begin{example}
  In Figure \ref{figone} in the introduction one may apply the above
  corollary to the vertex
  $v = 5$. This gives the separated ideal $I(T)$ of
  Example \ref{eks:intro-hept2}. The vertex $v=5$ is the only vertex we
  may use to get a separated ideal.
  These things may also be seen by Proposition \ref{pro:sepjoin-critsep}. 
  \end{example}

\subsection {Criteria for joining variables}

We present here basic results on dividing out a Stanley--Reisner
ring by variable differences. 

Let $X$ be a simplicial complex on a set $V$, and $F$ a facet of
$X$. Then for the algebraic set $A(X)$ in the affine space $\AA^V_k$ defined
by the Stanley--Reisner ideal $I_X \sus k[x_V]$,
the facet $F$ corresponds to the linear
space $A(F)$ in $\AA_k^V$ where all coordinates $x_v = 0$ for
$v \not \in F$, while
the $x_v$ for $v \in F$ may take arbitrary values.

For $v_1,v_2 \in V$, let $V_1 = (V \setminus \{v_1,v_2 \}) \cup \{v\}$.
The natural map $V \pil V_1$ sending $v_1, v_2 \mapsto v$ gives
a surjection of polynomial rings $k[x_V] \pil k[x_{V_1}]$. Let the ideal
$I$ be the image of $I_X$. Then $I$ may or may not be squarefree.
If $I$ is squarefree we say that $x_{v_1} - x_{v_2}$ {\em cuts squarefree}.
Then let $I = I_{X_1}$ where $X_1$ is the
associated simplicial complex. We then have a commutative diagram
of algebraic sets:
\[ \xymatrix{ A(X_1) \ar@{^{(}->}^{\phi}[r] \ar@{^{(}->}[d]
    &  A(X) \ar@{^{(}->}[d] \\
    \AA^{V_1}_k \ar@{^{(}->}[r] & \AA_k^V .}
\]
Let $e_v$ be the point in affine space $\AA^V_k$ where $x_v$ takes value
$1$ and the other variables value $0$. The map $\phi$ above sends
\begin{equation} \label{eq:regsimp:ae}
  \sum_{i \in V_1} a_i e_i  \mapsto a_v e_{v_1} + a_v e_{v_2} +
  \sum_{i \neq v} a_ie_i.
  \end{equation}

\begin{lemma} \label{lem:simp-nz}
Let $X$ be a simplicial complex on a set $V$.

\begin{itemize}
\item[a.] 
A variable difference $x_{v_1} - x_{v_2} $ where $v_1, v_2 \in V$
is a nonzero divisor for $S/I_X$
iff for each facet $F$ of $X$, at least one of the variables $v_1$ or $v_2$
is in $F$.
\item[b.] The ideal $I$ is squarefree iff whenever
  $F \cup \{v_1\}$ and $F \cup \{v_2\}$ are faces of $X$,
  then $F \cup \{v_1,v_2 \}$ is a face of $X$.
\item[c.] Let $F_1, \ldots, F_r$ be the facets of $X$.
  If the difference $x_{v_1} - x_{v_2}$ is a nonzero divisor and cuts squarefree,
  the facets of $X_1$ are
  $G_1, \ldots, G_r$ where:
  \begin{itemize}
  \item If\/ $F_i$ contains exactly one of $v_1$ and $v_2$,
    then $G_i = F_i \setminus \{v_1, v_2 \}$.
  \item If\/ $v_1, v_2$ are both in $F_i$ then
    $G_i = (F_i  \setminus \{v_1,v_2 \}) \cup \{v \}$.
  \end{itemize}
  \end{itemize}
\end{lemma}

\begin{proof}

  \noindent a. The associated primes of $I_X$ are
  the ideals generated by variables $(x_v)_{v \not \in F}$,
  one such ideal for each facet $F$. The variable difference is a nonzero
  divisor iff it is not in any of these ideals. This means that never
  both $v_1$ and $v_2$ are in such an ideal, or equivalently never both $v_1$ and $v_2$
  are outside of a facet $F$.

  \noindent b. The ideal $I$ is squarefree iff there is no minimal
  generator $x_{v_1}x_{v_2}x_F$ of $I_X$. But having such
  a generator means having faces $F \cup \{v_1 \}$ and
  $F \cup \{v_2\}$ but not a face $F \cup \{v_1, v_2 \}$.
  
  \noindent c. This follows by \eqref{eq:regsimp:ae} above.
\end{proof}

A sequence of linear forms $\ell_1, \ldots, \ell_r$ is a
regular sequence for $S/I_X$ iff for every facet $F$, it cuts down
$A(F)$ successively by one dimension for every $\ell_k$.

\begin{corollary} \label{cor:sepjoin-regseq}
Let $X$ be a simplicial complex on a set $V$.
Let $B$ be a forest on $V$, and denote $B_1, \ldots, B_m$ the trees
in $B$ and $V_i$ the support of $B_i$ for each $i$.
Then $\{ x_v-x_w \mid \{v,w \} \text{ edge of } B\}$ is a regular
sequence for $S/I_X$ iff for each facet $F$ and each $V_i$, at most one
of the vertices of $V_i$ is not in $F$.
\end{corollary}

\begin{proof}
A facet $F$ of $X$ gives the irreducible component $A(F)$ of the algebraic
set associated to $X$. When cutting down $A(F)$
by the sequence of variable differences associated to the edges of $B_i$
we have:
\begin{itemize}
\item If some $u \in V_i$ is not in $F$, the coordinate $x_v = x_u = 0$
for $v \in V_i$. This reduces dimension by $|V_i| - 1$.
\item If $V_i \sus F$, all coordinates $x_v$ for $v \in V_i$ become equal.
This again reduces dimension by $|V_i| - 1$.
\end{itemize}
Hence using the edges of the forest $B$ the linear space
$A(F)$ is cut down to a linear space whose dimension is
$\sum_i (|V_i|-1)$
less than $A(F)$ for each facet $F$ of $X$. But then the set of
variable differences is a regular sequence.

Conversely assume the sequence is regular. If there are $V_i$ and $F$
with $\{ v^\prime,v^{\prime \prime} \} \sus V_i \setminus F$,
then the variable differences $x_v - x_w$ 
associated to edges $\{v,w \}$ in $B_i$ would only give the restrictions
$x_v = 0$ for $v \in V_i \cap F$. This only cuts down dimension
at most by $|V_i| - 2$, contrary to the sequence being regular.
\end{proof}


\section{Stacked simplicial complexes}
\label{sec:triangulation}

Let $X$ be a simplicial complex on a set $A$. In the previous section
we used $V$ for the vertex set of $X$, but in the sequel we reserve $V$ for
the vertex set of the hypertree $T$ associated to the stacked
simplicial complex $X$.
In other words $V$ is an index set for the facets of $X$.

We show that $I_X$ may be successively separated to an ideal $I_{X^\prime}$,
where $X^\prime$ is a stacked simplicial complex of dimension two less
than the number of facets $|V|$. (See Figure~\ref{figfour}
for two examples of such $X^\prime$.)

\subsection{Stacked simplicial complexes and associated
  hypertree}
A facet $F$ of a simplicial complex $X$ is a {\it leaf} if there is a vertex
$v$ of $F$ such that $F$ is the only facet containing $v$. Such a vertex
is a {\it free vertex} of $X$. If $v$ is the only free vertex of
$F$ we say $F$ is {\it stacked} on $X_{-\{v\}}$.
(In general, a {\it free face} of a simplicial complex
is a face that is not a facet and that lies on exactly one facet, see for instance~\cite{Maund}. The term ``leaf'' is quite standard 
in graph theory, and  less common in the setting of simplicial complexes,
but see for instance~\cite{Far}.)

\begin{definition}
  A pure simplicial complex (i.e., where all the facets have the same dimension)
  is {\it stacked} if there is an ordering of its
  facets $F_0, F_1, \ldots, F_k$ such that if $X_{p-1}$ is the simplicial
  complex generated by $F_0, \ldots, F_{p-1}$, then $F_p$ is stacked on
  $X_{p-1}$.
\end{definition}

\begin{remark}
This is a special case of shellable simplicial complexes,
see~\cite[Subsection~8.2]{HeHi}. It is not the same as the notion of
simplicial complex being a tree as in~\cite{Far}, even if the tree
is pure. Rather the notion of stacked simplicial complex
is more general. For instance the triangulation of the heptagon
given in Example~\ref{eks:intro-hept1}, is not a tree in the
sense of \cite{Far}, since removing the triangles~$234$ and~$257$
one has no facet which is a leaf.
\end{remark}

\begin{remark} Stacked simplicial complexes are flag complexes.
  Every minimal nonface is an edge. Equivalently the Stanley--Reisner
  ideal is generated by quadratic monomials.
\end{remark}

A {\it hypergraph} is an ordered pair $H=(V,E)$ where $V$ is a set  and 
$E$ is a collection of subsets of $V$ such that no $e\in E$ is contained in
another $e'\in E$. The elements of $V$ are called
the vertices of $H$ and the elements of $E$ are called the edges of $H$.  
A hypergraph $H$ is a {\it hypertree} if 
\begin{enumerate}
\item[(i)] any two edges intersect in
either one or zero elements, 
\item[(ii)] $H$ is connected, i.e., for any two vertices $v$ and $w$
in $V$ there is 
a sequence $e_1,\dots,e_m$ of edges of $H$ with $v \in e_1$ and $w \in e_m$
and such that for every $i\in\{1,\dots,m-1\}$ one has
$e_i\cap e_{i+1}\ne\emptyset$, and
\item[(iii)] $H$ has no cycle, i.e., no
sequence of distinct vertices $v_0, v_1, \ldots, v_n$ save $v_n = v_0$, with
$n \geq 3$
such that each pair $\{v_{i-1},v_i\}$ is contained in an edge but
no triple $\{v_{i-1},v_i,v_{i+1}\}$ is contained in an edge.
\end{enumerate}
If $T^\prime$ and $T$ are hypertrees on the same vertex set,
$T^\prime$ is a {\it refinement} of $T$ if
\begin{enumerate}
\item[(i)] every edge of $T^\prime$
is contained in an edge of $T$, and 
\item[(ii)] every
 edge of $T$ is a union of edges of $T^\prime$.
 \end{enumerate}

\begin{definition}
  Let $X$ be a stacked simplicial complex with facets $F_v$
  indexed by a set $V$.
  We associate a hypertree
  to $X$ on the vertex set $V$. For each
  codimension-one face $G$, let $e_G = \{ v \in V \mid F_v \supseteq G \}$.
  The edge set of the hypertree is $E = \{ e_G \mid |e_G| \geq 2 \}$,
  the set of those $e_G$ containing at least two facets. 
\end{definition}

The simplicial complex $X$ is a triangulated ball iff its associated
hypertree $T$ is an ordinary tree, \cite[Theorem 11.4]{Bj}.
It can then be realized as a stacked polytope.
Such polytopes are extremal in the following sense: they
have the minimal number of faces, given the number of vertices (see \cite{Bar}).

\begin{obs}
  Let $X$ be a stacked simplicial complex which is a cone with $p$
  vertices in the cone apex. Thus $X$ is a join $X_1 * \Delta_p$
 where $\Delta_p$ is a simplex on $p$ elements and $X_1$ is not a cone.
 Then $X_1$ is also stacked and both $X$ and $X_1$ have
  the same associated hypertree. 
\end{obs}

    \subsection{Separating stacked simplicial complexes}
    \begin{lemma} \label{lem:tri-sep}
      Let\/ $X$ be a stacked simplicial complex of dimension $d$
      with hypertree~$T$.
      If\/ $T$ has $\geq d+3$ vertices, then using the procedure of
      Corollary~\ref{cor:sep-flag}, $X$ may be separated to
      a simplicial complex $X^\prime$ which is also stacked, and whose
      hypertree $T^\prime$ is a refinement of\/ $T$.
    \end{lemma}

    \begin{proof} Note that if $d = 0$, then
      $X$ is a collection of $ \geq 3$ vertices and the hypertree $T$ has one
      edge, the set of all facet indices $V$. By Corollary
      \ref{cor:sep-flag}, $X$ may be separated.

      Let $d \geq 1$ and $F_0, \dots, F_k$ be a stacking order for $X$. The facet
      $F_k$ is stacked on some previous facet $F_p$, with $p<k$.
      Let $v$ be the vertex of $F_p\setminus F_k$, $w$ the vertex of
      $F_k \setminus F_p$, and $G = F_p \cap F_k$. 
      Then $X_{-N(v)}$ contains $\{w \}$ as a component. If there are other
      components, we may apply Corollary \ref{cor:sep-flag}.

      Suppose then $\{w \}$ is the only component. Then $X_{-\{w\}}$ must be
      a cone
    over $v$. Let $Y$ be the link $\lk_X v$. It is stacked, of dimension
    $d-1$ and has $\geq d+2$ facets. By induction we may use Corollary
    \ref{cor:sep-flag} and separate $Y$, using an element $v^\prime$,
    to $Y^\prime$ whose tree is a
    refinement of that of $Y$. Let
    \[ V_1 \cup V_2 = (V \setminus \{v,w \}) \setminus N_Y(v^\prime)\]
    be the partition given in Corollary \ref{cor:sep-flag}. 

\medskip
\noindent 1. Suppose $N(v^\prime)$ (the neighborhood considered in $X$)
contains $w$.
Then $X_{-N(v^\prime)}$ has two components, supported on respectively $V_1$
and $V_2$, and we may apply Corollary
\ref{cor:sep-flag}.

\noindent 2. Suppose  $N(v^\prime)$ contains $G$ and not $w$.
Then $X_{-N(v^\prime)}$ has components $\{w \}$ together with  at least one
other component and we may again apply Corollary~\ref{cor:sep-flag}.

\noindent 3. Suppose $N(v^\prime)$ does not contain $w$ nor $G$.
Then there is $u \in G$ not in $N(v^\prime)$. So $u$ is in, say $V_1$.
Then $X_{-N(v^\prime)}$ may be written as a disjoint union $X_1  \cup X_2$,
with $X_2$ supported on $V_2$ and $X_1$ supported on $V_1 \cup \{w\}$.
Again we may apply Corollary~\ref{cor:sep-flag}.

Let us now show that $T^\prime$ is a refinement of $T$.
Let $G$ be a codimension-one face of $X$ contained in two or more facets
$F_i$ for $i \in D$, so  $D$ is an edge in $T$. Denote by $x_v$ the
variable used in the separation.
\begin{itemize}
\item If $G$ contains $v$, write $G = G^0 \cup \{ v\}$, and then $F_i$ is $F_i^0 \cup \{v\}$.
  Then $G^0 \cup \{v_1, v_2\}$ is a codimension-one face in
  facets $F_i^0 \cup \{ v_1, v_2 \}$ for $i \in D$. So $D$ is still
  an edge in $T^\prime$.
\item Suppose $G$ does not contain $v$. Let $D_1 \sus D$ index all
  $F_i$ in $\cF_1$ containing $G$ and similarly define $D_2$. 
  There might also be a facet $F = G \cup \{v\}$ containing $G$, in which
  case we extend both $D_1$ and $D_2$ with the index of this
  facet. Then $D_1$ and $D_2$ are edges of $T^\prime$ and they have at most
  one vertex in common.
\end{itemize}
\end{proof}

\begin{proposition} \label{pro:tria-unik}
  Let $X$ be a stacked simplicial complex of dimension $d$ which is
  {\em not} a cone,
  and let\/ $T$ be the associated hypertree.
  \begin{itemize}
  \item[a.] $T$ has $\geq d+2$ vertices,
\item[b.] If\/ $T$ has an edge of cardinality $\geq 3$, then $T$ has $\geq d+3$
  vertices
\item[c.] If\/ $T$ has $\geq d+3$ vertices, $X$ may be separated to a simplical
  complex $X^\prime$ whose tree $T^\prime$ is a refinement of\/ $T$.
\item[d.] If\/ $T$ is an (ordinary) tree with $d+2$ vertices, then
  $X$ is inseparable and the isomorphism class of $X$ is uniquely determind
  by $T$.
\end{itemize}
\end{proposition}

\begin{example}\label{ex:stackedsimpcomplfourfacets}
 For $d = 2$, Figure~\ref{figfour} 
 \begin{figure}
\begin{center}
\begin{tikzpicture}[dot/.style={draw,fill,circle,inner sep=1pt},scale=1.1]
\draw [help lines,white] (0,-.5) grid (4,2);
\draw [thick, fill=gray, opacity=0.3] (0,0)--(3,0)--(3.75,1.3)--(.75,1.3)--cycle;
\draw [thick,red] (.75,.43)--(1.5,.87)--(2.25,.43)--(3,.87);
\draw [thick] (0,0)--(3,0)--(3.75,1.3)--(.75,1.3)--cycle;
\draw [thick] (.75,1.3)--(1.5,0)--(2.25,1.3)--(3,0);
\fill (0,0)  circle (0.08);
\fill (1.5,0)  circle (0.08);
\fill (3,0)  circle (0.08);
\fill (.75,1.3)  circle (0.08);
\fill (2.25,1.3)  circle (0.08);
\fill (3.75,1.3)  circle (0.08);
\fill[red] (.75,.43)  circle (0.08);
\fill[red] (1.5,.87)  circle (0.08);
\fill[red] (2.25,.43)  circle (0.08);
\fill[red] (3,.87)  circle (0.08);
\end{tikzpicture}
\qquad\quad
\begin{tikzpicture}[dot/.style={draw,fill,circle,inner sep=1pt},scale=1.1]
\draw [help lines,white] (0,0) grid (3,2);
\draw [thick, fill=gray, opacity=0.3] (0,0)--(3,0)--(1.5,2.6)--cycle;
\draw [thick,red] (.75,.43)--(1.5,.87)--(2.25,.43);
\draw [thick,red] (1.5,.87)--(1.5,1.73);
\draw [thick] (0,0)--(3,0)--(1.5,2.6)--cycle;
\draw [thick] (.75,1.3)--(1.5,0)--(2.25,1.3)--cycle;
\fill (0,0)  circle (0.08);
\fill (1.5,0)  circle (0.08);
\fill (3,0)  circle (0.08);
\fill (.75,1.3)  circle (0.08);
\fill (2.25,1.3)  circle (0.08);
\fill (1.5,2.6)  circle (0.08);
\fill[red] (.75,.43)  circle (0.08);
\fill[red] (1.5,.87)  circle (0.08);
\fill[red] (2.25,.43)  circle (0.08);
\fill[red] (1.5,1.73)  circle (0.08);
\end{tikzpicture}
\caption{}
\label{figfour}
\end{center}
\end{figure}
 shows the two stacked
  simplicial complexes of dimension~$2$ with four facets. The corresponding
  trees are also drawn in red.
\end{example}

\begin{proof}[Proof of Proposition~\ref{pro:tria-unik}]
  \noindent {\em a,b.} Let $F_0, \ldots, F_k$ be a stacking order of facets.
  Let $X_p$ be the complex generated by $F_0, \ldots, F_p$. 
  Let   $C_p = \cap_{i = 0}^p F_i$ and $G_p$ be the codimension-one face of $F_p$
  which attaches it to $X_{p-1}$. Then for $p \geq 1$,
  $C_p=  \cap^p_{i = 1} G_i$ and $C_{p} = C_{p-1}  \cap G_{p}$.
  Note $G_{p}$ has codimension one in $F_{p-1}$ and $C_{p-1} \sus F_{p-1}$.
  But then $C_p$ has cardinality
  \[ |C_p| = |C_{p-1} \cap G_p| \geq |C_{p-1}| - 1. \]
  Since $|C_0| = d+1$ we get $|C_p| \geq d+1-p$ and so if $X$ is not
  a cone, $k \geq d+1$.  If $T$ has en edge of cardinality $\geq 3$,
  some $G_p$ equals some $G_r$ for $r < p$. Then
  $C_{p-1} \sus G_r \sus G_p$ and we get $C_{p-1} = C_p$. Thus
  $|C_q| \geq d+2-q$ for $q \geq p$, and so if $X$ is not a cone, $k \geq
  d+2$.

  \noindent {\em c.} This is shown in Lemma \ref{lem:tri-sep}.

  \noindent {\em d.} Let $X$ have associated tree $T$. Label the vertices of $T$ with
  $\{0,1, \ldots, d+1\}$.
  We assume the labeling is such that the
  induced subgraph on $[0,p]$ is always a tree for $p = 0, \ldots, d+1$.
  Then the corresponding ordering $F_0, F_1, \ldots,
  F_{d+1}$ of the facets of $X$ is a stacking order.

  Let $Y$ be another stacked simplicial complex with tree $S$
  isomorphic to $T$. Transferring the labeling from $T$, we get
  a stacking order $G_0, G_1, \ldots, G_{d+1}$ of the facets of $Y$.
  Let
\begin{eqnarray*}  F_{d+1}\setminus F_\ell = \{v \}, & \quad & 
G_{d+1} \setminus G_\ell = \{w \},
\end{eqnarray*}
where $\ell<d+1$ is such that $F_{d+1}$ is stacked on $F_\ell$.
The following restrictions are cones by part a, since they have $\leq d+1$
vertices
\[  X_{-\{v\}} = X^\prime * \{v^\prime\},  \quad
    Y_{-\{w\}} = Y^\prime * \{w^\prime \} \]
  and $X^\prime $ and $Y^\prime$ are not cones
  (since $X$ and $Y$ are not cones). Their trees are obtained
  from $T$ and $S$ by removing the vertices labeled $d+1$.
  The $F_i^\prime = \lk_{F_i} v^\prime$ for $i = 0,1,\ldots, d$
  form a stacking order for $X^\prime$
  and similarly the $G_i^\prime = \lk_{G_i}w^\prime$ form
  a stacking order for $Y^\prime$.

  By induction there is a bijection between $V \setminus \{v,v^\prime\}$
  and $W \setminus \{w, w^\prime \}$ sending the facet
  $F_i^\prime$ of $X^\prime$ to the facet $G_i^\prime $ of $Y^\prime$.
  Extend this to a bijection between $V$ and $W$  by
  $v \mapsto w, \, \, v^\prime \mapsto w^\prime$.
  Then the facet $F_i$ is sent to the facet $G_i$ for $i = 0, \ldots, d$.

  So consider the facets $F_{d+1}$ and $G_{d+1}$.
  Let the vertex $(d+1)$ of $T$ be attached to vertex $p \leq d$.
  So $F_{d+1}$ is attached by the codimension-one face $F_{d+1} \cap F_p$.
  But this is $F_{d+1} \setminus \{v\}$ and does not contain $v^\prime$
  ($F_{d+1}$ does not contain $v^\prime$ since $X$ is not a cone).
  So this codimension-one face is $F_p^\prime$. Similarly $G_{d+1}$ is
  attached to $G_p^\prime = G_{d+1} \setminus \{w\}$.
  Since $F_p^\prime$ is sent to $G_p^\prime$, the facet $F_{d+1}$ is sent
  to $F_{d+1}$. 
\end{proof}
 
\section{Trees and the associated separated model}
\label{sec:tree}

Given a tree $T$ we define the ideal $I(T)$. These ideals are the
separated models of stacked simplicial complexes.

Let $T$ be a tree whose set of vertices is $V$. Let $E = E(T)$ be
its set of edges. The {\it incidence relation} $\iC \sus E \times V$ is
the set of pairs $(e,v)$ such that $v \in e$. It comes
with a natural involution $\tau : \iC \pil \iC$ sending
$(e,v) \mapsto (e,w)$ where $e = \{v,w\}$.


For $v,w \in V$, denote by $vTw$ the unique path from $v$ to $w$
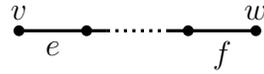
$$\begin{tikzpicture}[dot/.style={draw,fill,circle,inner sep=1pt},scale=.9]
\fill (0,0)  circle (.08);
\fill (1,0)  circle (.08);
\fill (2.5,0)  circle (.08);
\fill (3.5,0)  circle (.08);
\draw[very thick] (0,0)--(1.3,0);
\draw[very thick,dotted] (1.3,0)--(2.2,0);
\draw[very thick] (2.2,0)--(3.5,0);
\coordinate [label=above: $v$] (v) at (0,0);
\coordinate [label=below: $e$] (e) at (.5,0);
\coordinate [label=below: $f$] (f) at (3,0);
\coordinate [label=above: $w$] (w) at (3.5,0);
\end{tikzpicture}$$
and let $e,f$ be the edges incident to respectively $v,w$ on this path.
For a set~$A$ denote by $(A)_2$ the set of
subsets $\{a_1,a_2\}$ of cardinality~$2$.
From the directed tree~$T$ on $V$, we get a map
\begin{eqnarray*} \label{eq:tree:tomapA}
 \Psi:  (V)_2 & \pil & (\iC)_2 \\
\notag   \{v,w \} & \mapsto & \{(e,v),
                       (f,w) \}.
\end{eqnarray*}

For a graph $G$ on $V$ those vertices that are incident to
an edge of $G$ are called the {\it vertices of\/ $G$}.
The edges $\Psi(E(G))$ give a graph $\Psi G$ whose vertices are $\iC$. 

\begin{itemize}
\item 
If $G_1$ and $G_2$ have disjoint vertex sets, the same holds for
$\Psi G_1$ and $\Psi G_2$. 
\item If $G$ is a forest, then $\Psi G$ is a forest, since a cycle in $\Psi G$
  must come from a cycle in $G$.
\end{itemize}

The following is a basic object in this article.

\begin{definition} \label{def:tree-S}
  Let $\Spol$ be the polynomial ring whose variables are indexed by
  the incidence relation $\iC$. The {\it tree ideal\/ $I(T)$} in $\Spol$
  associated to the tree
  $T$ is the edge ideal of $\Psi T$. It is generated by
  the monomials $m_{v,w} = x_{e,v}x_{f,w}$, one monomial for each pair
  of distinct vertices $v,w$ in $V$.
  The edges $e$ and $f$ are incident to $v$ and $w$, respectively, 
  on the path $vTw$. 
\end{definition}

These tree ideals are introduced in \cite[Section 5]{AFL} (but
in a slightly less conceptual setting by indexing
the variables by $E \times \{0,1\}$). They are
shown to be all the possible separated models for the second power
$(x_e \, | \, e\in E(T))^2$
of the irrelevant maximal ideal  in the polynomial ring $k[x_e]_{ e\in E(T)}$
whose variables are indexed by the edges of $T$.
In particular the ideals $I(T)$ are Cohen--Macaulay and their graded Betti
numbers are precisely those of the graded free resolution of the
second power $(x_e \, | \, e \in E(T))^2$ of the graded maximal ideal
of $k[x_e]_{ e\in E(T)}$.


The following is given in \cite[Section 5]{AFL}.

\begin{lemma} \label{lem:tree:facets}
  The facets of the simplicial complex associated to the Stanley--Reisner 
  ideal\/ $I(T)$ are
  \[ F_v = \{ (e,w) \in \iC \mid w \text{ vertex on } e \text{ closest
      to } v \}, \]
  one facet for each vertex $v \in V$. The cardinality of these facets is
  then the number of edges of\/ $T$.
\end{lemma}

\begin{corollary}
  The ideal $I(T)$ defines the unique non-cone
  stacked simplicial complex with tree~$T$ of dimension $|E|-1$ with
  $|E|+1$ vertices, given in Proposition~\ref{pro:tria-unik}d.
  \end{corollary}


\medskip
A variation of the map $\Psi$ above is
$\oPsi = (\tau)_2 \circ \Psi$ where
$(\tau)_2: (\iC)_2 \pil (\iC)_2$ is derived from the involution
$\tau$ of $\iC$. 
Considering the path between $v$ and $w$
\[\begin{tikzpicture}[dot/.style={draw,fill,circle,inner sep=1pt},scale=1.1]
\fill (0,0)  circle (.08);
\fill (1,0)  circle (.08);
\fill (2.5,0)  circle (.08);
\fill (3.5,0)  circle (.08);
\draw[very thick] (0,0)--(1.3,0);
\draw[very thick,dotted] (1.3,0)--(2.2,0);
\draw[very thick] (2.2,0)--(3.5,0);
\coordinate [label=above: $v$] (v) at (0,0);
\coordinate [label=above: $v^\prime$] (vp) at (1,0);
\coordinate [label=below: $e$] (e) at (.5,0);
\coordinate [label=below: $f$] (f) at (3,0);
\coordinate [label=above: $w$] (w) at (3.5,0);
\coordinate [label=above: $w^\prime$] (wp) at (2.5,0);
\end{tikzpicture}\]
this variation is defined as
\begin{eqnarray*} \label{eq:tree:tomap}
 \oPsi:  (V)_2 & \pil & (\iC )_2 \\
\notag   \{v,w \} & \mapsto & 
                      \{(e,v^\prime), (f,w^\prime)\}.
\end{eqnarray*} 

We will divide the ring $\Spol/I(T)$ by the following variable differences:
\begin{definition} \label{def:tree:h} 
  For each pair $\{v,w\}$ in $(V)_2$ let  $h_{v,w}$ be the variable
  difference associated to the edge $\oPsi \{v,w\}$. So
  \[ h_{v,w} = x_{e,v^\prime}- x_{f,w^\prime}. \]
  \end{definition}
  Note that $h_{w,v} = -h_{v,w}$. Sometimes we write this as $h_e$ where
  $e = \{v,w\}$ when this sign plays no role.

\section{Regular quotients of tree ideals}
\label{sec:regtree}


We describe precisely what sequences of variable differences 
are regular for $\Spol/I(T)$. The combinatorial description is in
terms of partitions of the vertex set of~$T$, Theorem \ref{thm:regtree-part}.

\begin{definition} Let $T$ be an (undirected) tree with vertex set $V$. 
  \begin{itemize}
  \item The sequence of vertices $v,u,w$ is {\em $T$-aligned}
    if $u$ is on the path in $T$ linking $v$ and $w$.
  \item The set $\{v,u,w\}$ is {\em non-aligned} for $T$, if no ordering
    of them makes a $T$-aligned sequence.
  \end{itemize}
\end{definition}

\begin{example}
Consider the second tree in Figure~\ref{figfive}. The sequence of vertices $1,4,8$ is $T$-aligned, and the set $\{1,5,8\}$ is non-aligned for $T$.
\end{example}

Recall the variable difference $h_{v,w}$ from Definition \ref{def:tree:h}.
The variables of the polynomial ring $\Spol$ (see Definition \ref{def:tree-S})
are indexed by the incidence relation $\iC$.

\begin{lemma}\label{lem:reg-reg}
  The variable differences in $\Spol$ which are non-zero divisors
  for $\Spol/I(T)$ are those coming from the
  edges of\/ $\oPsi T$, i.e., the differences $h_{v,w}$.
\end{lemma}

\begin{proof}
  This is by Lemma  \ref{lem:simp-nz} and the description in
  Lemma \ref{lem:tree:facets} of the facets of the simplicial complex
  associated to $I(T)$. Given any edge outside of $\im \oPsi$, one
  may find a facet $F_v$ disjoint from this edge. 
\end{proof}

The following is the basic obstruction for a sequence of $h_{v,w}$'s
to be regular.

\begin{lemma} \label{lem:reg-nreg}
  Let $v,u,w$ be $T$-aligned. Then $h_{v,u}$ and $h_{v,w}$
  do not form a regular sequence.
\end{lemma}

\begin{proof}
  Let the path $vTw$ be:%
  
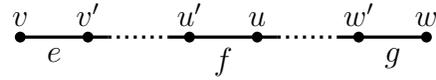
\[\begin{tikzpicture}[dot/.style={draw,fill,circle,inner sep=1pt},scale=.9]
\fill (0,0)  circle (.08);
\fill (1,0)  circle (.08);
\fill (2.5,0)  circle (.08);
\fill (3.5,0)  circle (.08);
\fill (5,0)  circle (.08);
\fill (6,0)  circle (.08);
\draw[very thick] (0,0)--(1.3,0);
\draw[very thick,dotted] (1.3,0)--(2.2,0);
\draw[very thick] (2.2,0)--(3.8,0);
\draw[very thick, dotted] (3.8,0)--(4.7,0);
\draw[very thick] (4.7,0)--(6,0);
\coordinate [label=above: $v$] (v) at (0,0);
\coordinate [label=below: $e$] (e) at (.5,0);
\coordinate [label=above: $v'$] (v') at (1,0);
\coordinate [label=above: $u'$] (u') at (2.5,0);
\coordinate [label=below: $f$] (f) at (3,0);
\coordinate [label=above: $u$] (u) at (3.5,0);
\coordinate [label=above: $w'$] (w') at (5,0);
\coordinate [label=below: $g$] (g) at (5.5,0);
\coordinate [label=above: $w$] (w) at (6,0);
\end{tikzpicture}\]
  We show that $h_{v,w}$ is not $\Spol/(I(T)+(h_{v,u}))$-regular,
  by showing that $x_{g,w}$ is in the colon ideal
  $(I(T)+(h_{v,u})):h_{v,w}$. Indeed
\begin{align*}
x_{g,w}h_{v,w}&=x_{g,w}(x_{e,v^\prime}-x_{g,w^\prime})\\
&=-x_{g,w}x_{g,w^\prime}+x_{g,w}x_{e,v^\prime}\\
              &=-x_{g,w}x_{g,w^\prime}+x_{g,w}x_{f,u^\prime}+x_{g,w}h_{v,u}
\end{align*}
is an element of $I(T)+(h_{v,u})$.
\end{proof}

The following is straight-forward.

\begin{lemma} \label{lem:reg-hhh}
  Let $\{ u,v,w\}$ be non-aligned. Then $h_{v,u} + h_{u,w} = h_{v,w}$.
\end{lemma}

%

\begin{definition} Let $T$ be a tree with vertex set $V$. Let $U \sus V$ and let
  $S$ be a tree on $U$ ($S$ is a priori unrelated to $T$). The tree $S$ {\em flows
  with $T$} if whenever $v,u,w$ are $T$-aligned vertices with $v,u,w \in U$,
  then $v,u,w$ are $S$-aligned.
\end{definition}

\begin{example} 
The tree $T$ in Figure~\ref{figfive:new} has black edges and
  seven vertices.
  The trees $S$ are drawn in red. In the first case
  $U = \{2,3,4,6\}$. The sequence of vertices $2,3,4$ is $T$-aligned but not $S$-aligned,
  so $S$ does not flow with $T$. In the second case $U = \{2,4,5,7\}$
  and $4,5,7$ is a $T$-aligned and $S$-aligned sequence. This tree $S$ flows with $T$.
\begin{figure}
\begin{center}
\begin{tikzpicture}[dot/.style={draw,fill,circle,inner sep=1pt},scale=1.2]
\draw[very thick,red] (2,0)--(3,1);
\draw[very thick,red] (3,1) .. controls (3.25,.7) and (3.25,.3) .. (3,0);
\draw[very thick,red] (3,0)--(4,1);
\draw[very thick] (1,0)--(5,0);
\draw[very thick] (3,0)--(3,1);
\draw[very thick] (4,0)--(4,1);
\fill (1,0)  circle (.07);
\fill (2,0)  circle (.07);
\fill (3,0)  circle (.07);
\fill (4,0)  circle (.07);
\fill (5,0)  circle (.07);
\fill (3,1)  circle (.07);
\fill (4,1)  circle (.07);;
\draw[thick,red] (2,0)  circle (.1);
\draw[thick,red] (3,0)  circle (.1);
\draw[thick,red] (3,1)  circle (.1);
\draw[thick,red] (4,1)  circle (.1);
\coordinate [label=below: $1$] (v) at (1,0);
\coordinate [label=below: $2$] (e) at (2,0);
\coordinate [label=below: $3$] (v') at (3,0);
\coordinate [label=below: $5$] (u') at (4,0);
\coordinate [label=below: $7$] (f) at (5,0);
\coordinate [label=above: $4$] (s) at (3,1);
\coordinate [label=above: $6$] (a) at (4,1);
\end{tikzpicture}\qquad\quad
\begin{tikzpicture}[dot/.style={draw,fill,circle,inner sep=1pt},scale=1.2]
\draw[very thick,red] (2,0)--(3,1)--(4,0);
\draw[very thick,red] (5,0) arc (0:180:.5 and .25);
\draw[very thick] (1,0)--(5,0);
\draw[very thick] (3,0)--(3,1);
\draw[very thick] (4,0)--(4,1);
\fill (1,0)  circle (.07);
\fill (2,0)  circle (.07);
\fill (3,0)  circle (.07);
\fill (4,0)  circle (.07);
\fill (5,0)  circle (.07);
\fill (3,1)  circle (.07);
\fill (4,1)  circle (.07);;
\draw[thick,red] (2,0)  circle (.1);
\draw[thick,red] (4,0)  circle (.1);
\draw[thick,red] (5,0)  circle (.1);
\draw[thick,red] (3,1)  circle (.1);
\coordinate [label=below: $1$] (v) at (1,0);
\coordinate [label=below: $2$] (e) at (2,0);
\coordinate [label=below: $3$] (v') at (3,0);
\coordinate [label=below: $5$] (u') at (4,0);
\coordinate [label=below: $7$] (f) at (5,0);
\coordinate [label=above: $4$] (s) at (3,1);
\coordinate [label=above: $6$] (a) at (4,1);
\end{tikzpicture}
\caption{}
\label{figfive:new}
\end{center}
\end{figure}
\end{example}

\begin{lemma} \label{lem:reg-edge}
  Let $U \sus V$ and let $S$ and $T$ be trees with vertex sets
  $U$ and $V$, respectively. Then $S$ flows with\/ $T$ iff 
whenever $\{v,w\}$ is an edge in $S$, there is no $u \in U\setminus
\{v,w\}$ such that $v,u,w$ are $T$-aligned.
\end{lemma}

\begin{proof}
  Let $S$ flow with $T$ and let $\{v,w\}$ be an $S$-edge.
  If there is $u$ such that $v,u,w$ are $T$-aligned,
  then $v,u,w$ would be $S$-aligned, which is not the case
  since $\{v,w\}$ is an edge in $S$.

  Conversely suppose the condition holds for edges in $S$.
  Let $v,u,w$ be vertices in $U$ which are $T$-aligned, so
  $\{v,w\}$ is not an edge of $S$. Suppose the path $vSw$
  does not contain $u$.
  We argue by induction on the length $\ell_S(v,w)$ of $vSw$
  that this is not possible. Since $\ell_S(v,w) \geq 2$
  let $r \in U$ on $vSw$ be distinct 
  from $v,w$ (note that $r \neq u$).
Then $\ell_S(v,w) > \ell_S(v,r)$ and $\ell_S(v,w) >\ell_S(r,w)$. 

  Consider in $T$ a path $p$ from $r$ to a vertex on the path $vTw$.
  We may assume
  only the end vertex of $p$ is on $vTw$.
  If $p$ first hits $vTw$ in the path segment $vTu$, then
  $r,u,w$ are $T$-aligned and with the path $rSw$ being such that
   $\ell_S(r,w) < \ell_S(v,w)$. By induction this situation is not possible.
  The case when $p$ first hits $vTw$ in $uTw$ is similar.
\end{proof}

\begin{corollary} \label{cor:regtree-exist}
  For any $U \sus V$, there is a tree $S$ with vertices $U$ flowing
  with $T$.
\end{corollary}

\begin{proof}
  Let $v \in V$. Consider $v$ as a center from which the tree $T$ branches out.
  Let $U_0$ be the subset of $U$ consisting of $w \in U$ such that the
  path $vTw$ contains no other vertex in $U$ than $w$ (in particular
  if $v \in U$ then $U_0 = \{v\}$).
  
  Now define $S$ to be the tree whose edges are:
  \begin{itemize}
    \item  Pairs $\{u,w\} \sus U$ 
  where i) $v,u,w$ are $T$-aligned
    (we allow $v = u$ if $v \in U$) and
    ii) the path $uTw$ intersects $U$ only in $\{u,w\}$.
  \item Give the vertices in $U_0$ a total order.
    If $u,w \in U_0$ are successive let $\{u,w\}$ be an edge in $S$.
\end{itemize}
  The tree $S$  fulfills the criterion of the lemma above, and
  hence flows with $T$.
\end{proof}

\begin{definition}
  If $S$ is a tree on the vertex set $U \sus V$, let $L(S)$ be the
  linear space with basis the $h_e = h_{v,w}$ where $e = \{v,w\}$
  are the edges of $S$.
If  $L(S)$ has a basis that is a regular sequence of
variable differences for $\Spol/I(T)$, we say that $L(S)$ 
is a {\em regular linear space}. (Equivalently some basis or any basis
of $L(S)$ is a regular sequence.) 
\end{definition}

\begin{lemma} \label{lem:reg-hvw}
  Let\/ $S$ be a tree with vertex set $U$, and assume that 
  only the end vertices $v$ and $w$ of the path $vTw$ are contained in $U$.
  \begin{itemize}
  \item If\/ $S$ flows with $T$, then $h_{v,w} \in L(S)$.
  \item If\/ $L(S)$ is a regular linear space, then $h_{v,w} \in L(S)$.
  \end{itemize}
\end{lemma}

\begin{proof}
  Let $v = v_0,v_1, \ldots, v_n = w$ be a path in $S$ of length $n \geq 2$.
  Then by assumption $v_1, \ldots, v_{n-1}$ are not in $vTw$.
  If the $v_i$-incident edges on $v_{i-1}Tv_i$ and
  $v_iTv_{i+1}$ are always distinct,
the paths would splice to give the unique path from $v$ to $w$. 
This cannot be the case since this path $vTw$ only has the end vertices in $U$.
 Hence for at least one $v_p, 1 \leq p \leq n-1$, these two 
  $v_p$-incident edges are equal.
 We have three possibilities:

\begin{itemize}
\item[i)] $v_p, v_{p-1},v_{p+1}$ are $T$-aligned,
\item[ii)] $v_p, v_{p+1},v_{p-1}$ are $T$-aligned,
\item[iii)] $\{v_{p-1},v_p,v_{p+1}\}$ is non-aligned for $T$.
\end{itemize}
For case i),~if $S$ flows with $T$, this would give that $v_p,v_{p-1},v_{p+1}$ are $S$-aligned,
which is not the case since $v_{p-1},v_p,v_{p+1}$ are $S$-aligned.
Similarly the second case ii)~is excluded.
If $L(S)$ is regular the first and second cases are aslo excluded by
Lemma~\ref{lem:reg-reg}. Hence only the last possibility iii)~is left.

If $S$ flows with $T$, then if $v_{p-1}Tv_{p+1}$ contains an element of $U$,
such an element would be either on $v_{p-1}Tv_p$ or on $v_pTv_{p+1}$.
But this is not the case by Lemma~\ref{lem:reg-edge}
since $v_{p-1},v_p$ and $v_p, v_{p+1}$ are edges in $S$.
Then we take out the edge  $\{ v_{p-1}, v_p\}$ from $S$ and take in
the edge $\{v_{p-1},v_{p+1}\}$ to get a new tree $S^\prime$ which still flows with~
$T$ by Lemma \ref{lem:reg-edge}.
By induction on the length $\ell_{S^\prime}(v,w)$, we have $h_{v,w} \in L(S)$.

If $L(S)$ is a regular linear space, then we again replace $S$ with $S^\prime$.
Due to Lemma \ref{lem:reg-hhh} we have $L(S) = L(S^\prime)$ and again we get
$h_{v,w} \in L(S)$.
\end{proof}

\begin{proposition} \label{pro:reg-SflowT}
  Let $S$ be a tree on $U \sus V$. If $L(S)$ is a regular linear space,
  then $S$ flows with $T$.
\end{proposition}

\begin{proof}
  Let $\{v,w\}$ be an edge in $S$, so $h_{v,w} \in S$.
  Suppose $v,u,w$ are $T$-aligned vertices in $U$. If we show this is
  not possible, then $S$ flows with $T$ by Lemma \ref{lem:reg-edge}.
  Choose $u$ as close as possible to $v$, so $vTu$ only contains $v$ and $u$
  from $U$. By Lemma~\ref{lem:reg-hvw}, $h_{v,u}$ is in $L(S)$.
  So both $h_{v,w}$ and $h_{v,u}$ are in $L(S)$. By Lemma~\ref{lem:reg-nreg},
  these two elements do not form a regular sequence, contradicting the fact that
  $L(S)$ is regular. Hence there can be no $u$ such that $v,u,w$ are
  $T$-aligned. So $S$ flows with~$T$.
  \end{proof}

\begin{lemma} \label{lem:reg-LRLS}
For any two  trees $R$ and $S$ on $U$ flowing with $T$, one has $L(R) = L(S)$.
Thus $U$ determines a {\em unique} regular linear space,
denoted $L(U)$.
\end{lemma}

\begin{proof}
  Let $\{v,w\}$ be an $R$-edge. We show $h_{v,w} \in L(S)$.
Since $R$ flows with $T$, the path $vTw$ does
not contain any elements of $U$ save the end vertices.
Since $S$ flows with $T$, Lemma \ref{lem:reg-hvw} gives $h_{v,w} \in L(S)$.
\end{proof}

\ignore{
Let $(V)_2$ denote the pairs of elements of $V$. This is the set
of edges in the complete graph on $V$. Given an oriented tree $T$ on $V$, we
then get a map
\begin{eqnarray} \label{eq:rg:tomap}
 \oPsi:  (V)_2 & \pil & (E(T) \times \{ 0, 1 \})_2 \\
\notag   \{v,w \} & \mapsto & \{(e,\overline{e_{to}(w)},
                       (f,\overline{e_{to}(v)}) \}
\end{eqnarray}
(Note that this map presupposes an oriented tree $T$ on $V$. A different
orientation on $T$ gives an isomorphic $\oPsi$, which nevertheless give
the same map when composing with the projection to $(E(T))_2$.
Also there is no natural map from $V$ to $E(T) \times \{0,1\}$.)

So given a graph $G$ on $V$ the edges $\oPsi(E(G))$ give a graph $\oPsi G$.
\begin{itemize}
\item 
If $G_1$ and $G_2$ have disjoint vertex sets, the same holds for
$\oPsi G_1$ and $\oPsi G_2$. 
\item If $G$ is a forest, then $\oPsi G$ is a forest, since a cycle in $\oPsi G$
  must come from a cycle in $G$.
\end{itemize}
}

\begin{lemma} \label{lem:reg-graph}
  Let $G$ be a graph on vertex set $V$ (with $G$
  a priori unrelated to $T$).
\begin{itemize}
\item[a.] If $\{h_e\}_{e\in G}$ is a
regular sequence for $\Spol/I(T)$, then $G$ is a forest.
\item[b.] If\/ $G$ is a forest consisting of the trees $S_1, \dots, S_r$,
then $\{h_e\}_{e \in G}$ is a regular sequence iff each
$\{ h_e\}_{e \in S_i}$ is a regular sequence.
\end{itemize}
\end{lemma}

\begin{proof}
  \noindent {\em a.} It is enough to show that if $G$ is a cycle $O$
  then $\{h_e\}_{e \in O}$ is not a regular sequence.
    Denote by $L(O)$ their linear span, and let the cycle be
    $v_0, v_1, \ldots, v_{n-1}, v_n = v_0$ of length $n$.


  We now use induction on the length $n$ of the cycle to show that
  $L(O)$ cannot be regular linear space. 
  Not every sequence $v_{i-1}v_iv_{i+1}$ is  $T$-aligned for $i = 1, \ldots n-1$
  since $v_0 = v_n$. Suppose $v_{p-1}v_pv_{p+1}$ is not $T$-aligned.
  If say $v_pv_{p-1}v_{p+1}$ are $T$-aligned then $h_{v_p,v_{p-1}}$ and
  $h_{v_p,v_{p+1}}$ do not form a regular sequence, against the assumption.
  By the same reason $v_pv_{p+1}v_{p-1}$ are not $T$-aligned. Hence
  $\{v_{p-1},v_p,v_{p+1}\}$ is non-aligned. Then $h_{v_{p-1},v_{p+1}}$
  is $h_{v_{p-1},v_p} + h_{v_p,v_{p+1}}$. Take the edges $\{v_{p-1},v_p\}$ and
  $\{v_p,v_{p+1}\}$ out from the cycle $O$ and take in the edge
  $v_{p-1},v_{p+1}$ to make a new cycle $L(O^\prime) \sus L(O)$.
  By induction $L(O^\prime)$ is not a regular linear space
  and so neither is $L(O)$.
  
  \noindent {\em b.} Suppose each $S_i$ gives a regular sequence. This sequence
  is determined by the edges of $\oPsi S_i$, and this is a forest.
  By Corollary \ref{cor:sepjoin-regseq} this is equivalent to each tree in
  $\oPsi S_i$ giving a regular sequence. But the disjoint
  union of the trees in the
  $\oPsi S_i$ are precisely the trees
  in $\oPsi S$. Hence Corollary \ref{cor:sepjoin-regseq} gives the result.
\end{proof}

The following is the converse of Proposition \ref{pro:reg-SflowT}:
 \begin{proposition} Let $S$ be a tree on $U \sus V$.
  If $S$ flows with $T$, then $L(S)$ 
  is a regular linear space.
\end{proposition}

\begin{proof}
  By Lemma  \ref{lem:reg-LRLS} above, if $U$ is the vertex set of $S$,
  we may choose $S$ to be any tree on $U$ that flows with $T$.
  
  Let $v \in V$. Consider the face
  \[ F_v = \{ (e,w^\prime) \mid w^\prime \text{ vertex on }
    e \text{ closest to } v\}. \]
Let $U \sus V$ and define
the tree $S$ flowing with $T$ with vertices in $U$ 
as in Corollary~\ref{cor:regtree-exist}. This tree comes
with two types of edges:
%
\begin{itemize}
\item Edges $\{u,w\}$ where $u$ and $w$ are two successive elements in the ordering of $U_0$.
  Then $\oPsi \{u,w\} = \{(f,u^\prime), (g,w^\prime)\}$
  where $f = \{u,u^\prime \}$ is the edge on $uTv$ going
  out from $u$ and similarly $g = \{w,w^\prime\}$
  the edge on $vTw$ going out form $w$.
  These $\{(f,u^\prime),(g,w^\prime)\}$ give a tree $T_0$ with vertices
  from the incidence relation $\iC$
  (actually a line graph). 
\item Edges $\{u^\prime,w\}$ where $u^\prime$ is the element in $U$
  on the path $vTw$ closest to $w$.
  If $f = \{w^\prime, w\}$ and $g = \{u^\prime, u \}$ are the edges on
$u^\prime T w$, one has $\oPsi \{u^\prime,w\} = \{(g,u),(f,w^\prime)\}$.
 Each such $w$ gives a unique $u^\prime$, but
 one $u^\prime$ may correspond to several $g$'s and $w$'s. For each pair
 $u^\prime,g$ these edges form a tree $T_{u^\prime,g}$, a star,
 with vertices from $\iC$.
\end{itemize}
The trees $T_0$ and $T_{u^\prime,g}$ (with vertices from $\iC$)
are all disjoint. Together the edges of these trees give all variable
differences $h_{v,w}$ for $\{v,w \}$ an $S$-edge.
The vertices of $T_0$ are contained in $F_v$. Each $T_{u^\prime,g}$ has all its
vertices save $(g,u)$  contained in $F_v$.
By Corollary~\ref{cor:sepjoin-regseq},
the linear space $L(S)$ is regular.
\end{proof}

\begin{theorem} \label{thm:regtree-part}
There is a one-to-one correspondence between regular linear spaces
for $\Spol/I(T)$ and partitions $Q$ of the vertex set\/ $V$.
If the partition of\/ $V$ is\/ $Q : U_0 \sqcup U_1 \sqcup \cdots \sqcup U_r$, then this
regular linear space is
\[ L(Q) = L(U_0) \oplus \cdots \oplus L(U_r). \]
\end{theorem}

\begin{proof}
By Lemma \ref{lem:reg-LRLS}, each $U_i$ determines a unique linear space
$L(U_i)$. If $S_i$ is a tree on $U_i$ flowing with $T$, then $L(S_i) = L(U_i)$.
Let $S = \cup_{i=0}^rS_i$. By Lemma~\ref{lem:reg-graph}, the edges of $\oPsi S$ give a regular
sequence. This regular sequence is a basis for $L(Q)$.

Conversely if $L$ is a regular linear space generated by the regular
elements $\{h_e \}_{e \in G}$ for some graph $G$ on $V$,
by Lemma \ref{lem:reg-graph} the graph $G$ decomposes into a forest and
we get a partition of $V$ where
each $U_i$ is the vertex set of each tree in the forest.
(The vertices $v$ of $V$ not incident to any edge of $G$ give
singletons $\{v\}$ in the partition.)
\end{proof}

\begin{corollary}
  The length of the longest regular sequence of variable differences
  for\/ $\Spol/I(T)$ is $|E|$. Such a sequence
corresponds
to the trivial partition of\/ $V$ with only one part, the set\/ $V$ itself.
The corresponding
tree that flows with $T$ is just $T$ itself. Hence this regular
sequence is given by $\{h_e\}_{e \in E}$ and the quotient ring
is $k[x_E]/(x_e \, | \, e \in E)^2$.
\end{corollary}

\section{Squarefree quotients}
\label{sec:regsqfree}
We determine what regular linear spaces give quotient rings of $\Spol/I(T)$
whose associated ideals are squarefree.
These are the Stanley--Reisner rings of stacked simplicial complexes.
Let $T$ be a tree with
vertices $V$.

\begin{lemma} \label{lem:regsq-ind} Let $U \sus V$ 
and let $S$ be a tree on $U$.
If\/ $S$ flows with\/ $T$, then the regular quotient
of\/ $\Spol/I(T)$ by $\{h_e \}_{e \in S}$ is
a squarefree monomial ideal iff the vertex set $U$ is
an independent vertex set in $V$ for the tree\/ $T$.
\end{lemma}

\begin{proof} The following is essential to note:
  The variables in the quotient ring modulo the sequence
  $\{h_e\}_{e \in S}$ correspond precisely to
  the connected components of the graph $\oPsi S$
  with vertex set the incidence relation $\iC$.

  If the vertex set $U$ is dependent, say contains end vertices of an
  edge $e = \{v,w\}$, then we divide out by $x_{e,v} - x_{e,w}$ and
  the ideal of the quotient ring will contain $x_e^2$
  as a generator and so is not squarefree.
  
Suppose then that $U$ is independent.
Let $\{v,w \}$ be a pair of vertices in $U$. 
  Suppose the associated monomial $x_{e,v}x_{f,w}$ becomes
    a square after dividing out by the regular sequence.
    This means that $ (e,v)$ and $(f,w)$ are in the same
    connected component of $\oPsi S$. Let the edge $e$ have vertices
    $v,v^\prime$ and
    the edge $f$ vertices $w^\prime, w$.
    So $v^\prime$ and $w^\prime$ are on the path
    $vTw$. Removing the edge $e$ from $T$ we get a component
    $T_v$ containing $v$,
    and similarly removing $f$ from $T$ we get a component $T_w$ containing $w$.

    Any edge in $\oPsi S$ containing $(e,v)$ is the image of an edge
    $\{\overline{v},v^\prime \}$ in $S$ where $\overline{v} \in T_v$.
    Similarly we have an edge $\{w^\prime,  \overline{w} \}$ in $S$
    where $\overline{w} \in T_w$. But since $(e,v)$ and
    $(f,w)$ are in the same connected component of $\oPsi S$,
    there must in $S$ be an edge $\{ \tilde{v}, \tilde{w} \}$ where $\tilde{v}$
    is in $T_v$ and $\tilde{w}$ is in $T_w$.
    Then either $v^\prime$ or $w^\prime$ from $U$  is
    in the interior of the path $\tilde{v}T \tilde{w}$.
    Since $S$ flows with $T$ this cannot be the case by
    Lemma~\ref{lem:reg-edge}.  \end{proof}

\begin{theorem} \label{thm:regsq-partind}
There is a one-to-one correspondence between regular linear spaces
for\/ $\Spol/I(T)$ giving squarefree quotient rings, and partitions of\/
$V$ into sets of independent vertices.
\end{theorem}

\begin{proof}
  Suppose we have a squarefree quotient ring. 
  Each part
  $U_i$ of the partition gives a regular linear space $L(U_i)$.
  By Lemma~\ref{lem:regsq-ind}, $U_i$ is independent.
  Conversely, if we have a partition of $V$ into independent sets $U_i$,
  let $S_i$ be a tree on $U_i$ flowing with $T$. 
  The images $\oPsi S_i$ have disjoint vertex sets as $i$ varies.
  Lemma \ref{lem:regsq-ind} 
  above shows that the quotient is squarefree.
  \end{proof}

  Using Theorem \ref{thm:part-VE},
  the above may equivalently be formulated as follows:
  
  \begin{theorem} \label{thm:regsq-partedge}
    There is a one-to-one correspondence between regular linear spaces
    for\/ $\Spol/I(T)$ giving squarefree quotient rings, and partitions of
    the edge set\/ $E$.
  \end{theorem}

   If $P$ is a partition of the edge set corresponding to the
   partition $Q$ into independent vertex sets, write
   $L(P) = L(Q)$.

\begin{corollary}
The length of the longest regular sequence of variable differences
giving a squarefree quotient of\/ $\Spol/I(T)$ is
$|E| - 1$. It corresponds to the unique partition of\/ $V$ into two
independent sets of\/ $V$ for the tree\/ $T$. Thus the associated regular
linear space is also unique.
\end{corollary}

\section{Partial order on partitions}
\label{sec:part2}

   If $Q$ and $Q^\prime$ are partitions of the vertex set $V$ of a tree $T$,
   we get the linear
   spaces $L(Q)$ and $L(Q^\prime)$. What does the inclusion relation on
   linear spaces correspond to on partitions? Since the linear spaces depend
   on additional
   structure coming from the tree $T$, this is not simply refinement of
   partitions.

\subsection{Partitions of the vertex set}

\begin{definition} Let $U^\prime \sus U \sus V$. Then $U^\prime$ is
{\em convex} in $U$ if for every $v,w \in U^\prime$, all vertices on the
path $vTw$ that are contained in $U$ are  in $U^\prime$.
\end{definition}

Note that such a $U^\prime$ may be convex in some $U$ while not being convex in $V$.

\begin{lemma} Let\/ $U^\prime$ and\/ $U$ be subsets of\/ $V$. If $L(U^\prime)$
  is a subspace of\/ $L(U)$ then \/ $U^\prime$ is a convex subset of\/ $U$,
  or $U^\prime$ is a singleton (then $L(U^\prime) = 0$).
Conversely if\/ $U^\prime \sus U$ is a convex subset, then $L(U^\prime)$
is a subspace of\/ $L(U)$. 
\end{lemma}

\begin{proof} Suppose $L(U^\prime)$ is a nonzero subspace of $L(U)$
  and there exists
  $v \in U^\prime \setminus U$.  There is another $w \in U^\prime$ such
  that $h_{v,w} \in L(U^\prime)$. 
Consider the path $vTw$ in $T$:
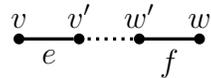
\[\begin{tikzpicture}[dot/.style={draw,fill,circle,inner sep=1pt},scale=.8]
\fill (0,0)  circle (.09);
\fill (1,0)  circle (.09);
\fill (2,0)  circle (.09);
\fill (3,0)  circle (.09);
\draw[very thick] (0,0)--(.91,0);
\draw[very thick,dotted] (1,0)--(2,0);
\draw[very thick] (3,0)--(2,0);
\coordinate [label=above: $v$] (v) at (0,0);
\coordinate [label=below: $e$] (e) at (.5,0);
\coordinate [label=above: $v'$] (v') at (1,0);
\coordinate [label=below: $f$] (f) at (2.5,0);
\coordinate [label=above: $w$] (w) at (3,0);
\coordinate [label=above: $w^\prime$] (wp) at (2,0);
\end{tikzpicture}\]
Then $h_{v,w}  = x_{e,v^\prime} - x_{f,w^\prime}$ is in $L(U^\prime)$.
If the edge $e$ occured in some $h_{a,b}$
generating $L(U)$, since $v \not \in U$,
one of $a$ or $b$ would have to be $v^\prime$, and
$v^\prime \in U$. But then this $h_{a,b}$ would contain $x_{e,v}$ instead
of $x_{e,v^\prime}$. Hence $U^\prime \sus U$. 

Let us show that $U'$ is convex in $U$. Let $v,w \in U^\prime$ be such that $vTw$ contains some $u \in U \setminus
U^\prime$. By possibly moving $v$ and $w$ closer to $u$, and $u$ closer
to $v$, we may assume on $vTw$ that
$v$ and  $w$ are the only vertices in $U^\prime$, and on $vTu$ that $v$ and $u$
are the only
vertices in $U$. But then $h_{v,w} \in L(U^\prime)$ and $h_{v,u}\in L(U)$ by Lemma \ref{lem:reg-hvw}. If $L(U^\prime) \sus L(U)$ this could not
be the case by Lemma \ref{lem:reg-nreg},
since $L(U)$ is regular. Hence, if we have inclusion,
$U^\prime$ must be convex in $U$. 

Conversely if $U^\prime$ is convex in $U$, then letting $S^\prime$ be
a tree on $U^\prime$ flowing with $T$, by Lemma \ref{lem:reg-hvw},
for each edge $e$ in $S^\prime$ we have $h_e \in L(U)$.
\end{proof}

The following is immediate from the above.

\begin{theorem}
Let $Q^\prime$ and\/ $Q$ be partitions of\/ $V$. Then
$L(Q^\prime) \sus L(Q)$ iff each part\/ $U_i$ of\/ $Q$ is a union of 
parts of\/ $Q^\prime$ which are convex for $U_i$.
Write then $Q^\prime \ord Q$.   
\end{theorem}

In a partition $Q$ of $V$, if $U_i$ and $U_j$ are parts
such that either $U_i$ or $U_j$ is not convex in $U_i \cup U_j$,
we say that $U_i$ and $U_j$ are {\em intertwined}.

\begin{corollary}
The maximal partitions for the partial order $\ord$ are the
partitions $Q$ such that any two parts $U_i$ and $U_j$ in the
partition are intertwined.
\end{corollary}

\begin{example}
  In the introduction, looking at Figure \ref{figthree}, the partition of
  vertices in Example \ref{eks:intro-hept3} is not maximal. We may join
  \[ \{1,5\} \cup \{2\} \cup \{3\} \cup \{4\} \ord \{1,5\} \cup \{2,3,4\}.\]
  The latter vertex partition is maximal since it is intertwined.
  Also note that the first partition
  is not $\ord \{1,2,3,4,5\}$, since $\{1,5\}$ and $\{2\}$ (as well as $\{4\}$)
  are intertwined.

  The partition $\{1,2,3,4,5\}$ corresponds to the quotient ring
  $k[x_E]/(x_e)^2_{e \in E}$, which is $k[x_E]$ divided by the square of the
  maximal graded ideal.
  Hence this ring is {\em not} a quotient ring of
  $\Spol/(I(T) + (x_{a,2} - x_{d,4}))$
  of Example \ref{eks:intro-hept2}, by a {\em regular linear space}.
  (But it is of course a quotient taking a suitable general linear space.)
  \end{example}

\begin{corollary}
Let\/ $\overline{V}$ be the partition of\/ $V$ into singletons. Then for any
partition $Q$ of\/ $V$ the interval\/ $[\overline{V},Q]$ with respect to the partial order
$\ord$ is a Boolean lattice.
\end{corollary}

\begin{proof}
Given a subset $U$ of $V$, we must show that the lattice of partitions
of $U$ into convex parts is a Boolean lattice. Let $v$ be extremal in
$U$ in the sense that every other vertex of $U$ is on the same side of $v$,
i.e., there is an edge $e = \{v,w\}$  from $v$ such that the path from $v$ to
any other vertex of
$U$ starts with the edge $e$.  
Let $U^\prime = U \setminus \{ v\}$. By induction the
lattice of partitions of $U^\prime$ into convex subsets is a Boolean lattice
$B$. The partitions $Q$ of $U$ into convex subsets are now of two types:
either $\{v\}$ is a singleton class, or $v$ and $w$ are in the same class.
This gives that the lattice of partitions of $U$ identifies as
$B \times \{0,1\}$ and so is Boolean.
\end{proof}

\subsection{Partitions of the edge set}

If $D^\prime \sus D \sus E$ are sets of edges of $T$, we may as
above define the notion of $D^\prime$ being convex in $D$.
As above we may show:

\begin{proposition}
Let\/ $P^\prime$ and\/ $P$ be partitions of\/ $E$. Then
$L(P^\prime) \sus L(P)$ iff each part\/ $E_i$ of\/ $P$ is a union of parts of\/
$P^\prime$ which are convex for\/ $E_i$. 
Write then\/ $P^\prime \ord P$.   
\end{proposition}

\begin{corollary}
The maximal partitions for the partial order $\ord$ are the
partitions $P$ such that any two parts $E_i$ and $E_j$ in the
partition are intertwined.
\end{corollary}

\begin{example}
  In the introduction, looking at Figure \ref{figthree}, the partition of
  edges in Example \ref{eks:intro-hept3} is not maximal. We may join
  \[ \{a,d\} \cup \{b\} \cup \{c\} \ord \{a,d\} \cup \{b,c\}.\]
  In the latter partition the parts are intertwined and
  so it is maximal. It correspond to the vertex partition
  $\{1,5\} \cup \{3,4\} \cup \{2\}$. This vertex partition is also
  maximal (but that does not necessarily follow from the edge partition
  being maximal).
  \end{example}

\begin{corollary}
Let\/ $\overline{E}$ be the partition of\/ $E$ into singletons. Then for any
partition $P$ of\/ $E$ the interval\/ $[\overline{E},P]$ with respect to the partial order
$\ord$ is a Boolean lattice.
\end{corollary}

\section{Hypertree of quotients and triangulated balls}
\label{sec:ball}

We describe the squarefree quotients of $\Spol/I(T)$ by regular linear spaces
whose associated simplicial complex is a triangulated ball.
In particular we describe when we get triangulations of polygons.

Let
\[ P : E_1 \sqcup E_2 \sqcup \cdots \sqcup E_r \]
be a partition of the edge set $E$ of the tree $T$. We may think of
the edges of $E_i$ as a color class.
The partition corresponds
    by Theorem~\ref{thm:part-VE} to
    a partition $V = U_0 \sqcup U_1 \sqcup \cdots \sqcup U_r$
    of the vertex set into independent sets of vertices.
    Let $S = S_0 \cup  S_1 \cup \cdots \cup S_r$ where
    the $S_i$ are trees on $U_i$ flowing with $T$.
    The image $\oPsi S $ is a forest and each $\oPsi S_j$ is a collection
    of connected components (trees) of $\oPsi S$.
    Moreover $L(P) = \oplus_{i=1}^r L(S_i)$.
    In the sequel we also write $\oPsi P$ for $\oPsi S$.

    \medskip

Let us describe the variables in the quotient ring $\Spol/L(P)$ (this is
a polynomial ring). These variables identifiy as subsets of
of the incidence relation $\iC$.
Those subsets which contain more than one element arise as follows.
For each class $E_i$
consider maximal sets of edges $E_{ij} \sus E_i$ such that for every pair
of edges $f,g$ in $E_{ij}$, the only edges in $E_i$ on the unique path from
$f$ to $g$ are $f$ and $g$ themselves. For given $i$ two such maximal $E_{ij}$
and $E_{ij^\prime}$ have at most one edge in common. (In fact the $E_{ij}$'s
form the set of edges in a hypertree on $E_i$.)
If $f,g$ are edges in a $E_{ij}$ with path
  
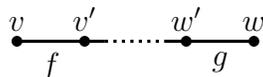
\[\begin{tikzpicture}[dot/.style={draw,fill,circle,inner sep=1pt},scale=.9]
\fill (0,0)  circle (.08);
\fill (1,0)  circle (.08);
\fill (2.5,0)  circle (.08);
\fill (3.5,0)  circle (.08);
\draw[very thick] (0,0)--(1.3,0);
\draw[very thick,dotted] (1.3,0)--(2.2,0);
\draw[very thick] (2.2,0)--(3.5,0);
\coordinate [label=above: $v$] (v) at (0,0);
\coordinate [label=below: $f$] (e) at (.5,0);
\coordinate [label=below: $g$] (f) at (3,0);
\coordinate [label=above: $w$] (w) at (3.5,0);
\coordinate [label=above: $v^\prime$] (vp) at (1,0);
\coordinate [label=above: $w^\prime$] (wp) at (2.5,0);
\end{tikzpicture}\]
then $x_{f,v^\prime} - x_{g,w^\prime}$ is
a variable difference in $L(P)$. It gives a class
$[x_{f,v^\prime}]$, a variable in $\Spol/L(P)$. 
This gives one variable in $\Spol/L(P)$ for each set $E_{ij}$.

\begin{example}\label{ex:threecolors}
 In Figure~\ref{figsix}
 we have a partition of the edges into three
  color classes.
  The four red edges give eight red variables in $\Spol$.
  The red edges give two maximal sets $\{a,c\}$ and $\{c,e,g\}$, each of which
  combines into one variable, giving five red variables in the
  quotient ring $\Spol/L(P)$. 
  \end{example}
  
\begin{figure}
\begin{center}
\begin{tikzpicture}[dot/.style={draw,fill,circle,inner sep=1pt},scale=1.4]
\draw[very thick,red] (3,.9)--(3.7,0)--(4.6,0)--(4.6,-.9);
\draw[very thick,blue] (3.7,0)--(3,-.9);
\draw[very thick,green] (5.5,0)--(4.6,0)--(4.6,.9);
\draw[very thick,blue] (5.5,0)--(6.2,-.9);
\draw[very thick,red] (5.5,0)--(6.2,.9);
\fill (3,.9)  circle (0.07);
\fill (3,-.9)  circle (0.07);
\fill (3.7,0)  circle (0.07);
\fill (4.6,0)  circle (0.07);
\fill (4.6,.9)  circle (0.07);
\fill (4.6,-.9)  circle (0.07);
\fill (5.5,0)  circle (0.07);
\fill (6.2,.9)  circle (0.07);
\fill (6.2,-.9)  circle (0.07);
\node at (3.15,.45) {$a$};
\node at (3.15,-.47) {$b$};
\node at (4.15,.17) {$c$};
\node at (4.75,.47) {$d$};
\node at (4.45,-.5) {$e$};
\node at (5.05,-.2) {$f$};
\node at (6,.41) {$g$};
\node at (6.05,-.47) {$h$};
\node at (5.7,0) {$v$};
\end{tikzpicture}
\caption{}
\label{figsix}
\end{center}
\end{figure}

  We now describe the facets of the simplicial complex corresponding
  to the quotient $\Spol/(I(T) + L(P))$. 
For each $v \in V$ and color class $E_i$, let $E_i^v$ be the set of edges $f$ in $E_i$ such
that on the path from $f$ to $v$ the only edge in $E_i$ is $f$ itself.
Then $E_i^v$ is a maximal set $E_{ij}$ as above and hence gives
a variable $x_{E_i^v}$ in $\Spol/L(P)$. We have $x_{E_i^v} = x_{E_j^w}$ iff
$i = j$ and there is no edge from $E_i$ on the path from $v$ to $w$.
Let
\[ F_v = \{ E_i^v \mid i = 1, \ldots, r \}. \]

\begin{example}
  Consider Figure~\ref{figsix}. The facet $F_v$ of $\Spol/(I(T) + L(P))$ is
  of cardinality~$3$. Its elements are the three maximal sets
  \[E^v_\text{red} = \{c,e,g\}, \quad E^v_\text{blue} = \{h,b\}, \quad
  E^v_\text{green} = \{f \}. \]
\end{example}

\begin{lemma}
    The facets of the simplicial complex associated to $\Spol/(I(T)+L(P))$
  are the $F_v$'s, for $v \in V$.
  In particular the cardinality of each facet is the number of classes
  in the partition~$P$.
\end{lemma}

\begin{proof}
  This follows by repeated use of Lemma \ref{lem:simp-nz}.
\end{proof}

\begin{lemma} \label{lem:ball-fe}
  Let $e = \{v,w\}$ be an edge in $T$, in the class $E_k$.
Then $E_i^v = E_i^w$ for $i \neq k$. 
  The facets $F_v$ and $F_w$ have a codimension-one face in common.
  It is the set
  \[ F_e = \{ E_i^v (= E_i^w) \mid i \neq k \}. \]
  \end{lemma}

\begin{proof}
  This is clear.
\end{proof}

\begin{lemma} \label{lem:ball-cod} The facets $F_v$ and
      $F_w$ have a codimension-one face $G$ in common if and only if
      the path from $v$ to $w$ has all edges of the same color. Then for all
      edges $e$ on this path, the $F_e$ are equal, and this is $G$.
      In particular $G$ is common to all facets $F_u$ for $u$ on
      this path.
\end{lemma}

\begin{proof}
  Suppose the edges on the path are all of the same color red.
Let the path be $v = u_0, u_1, \ldots, u_m = w$ with $e_i$ the edge
$\{u_{i-1},u_i\}$. Then for each edge $e_i$
\[ F_{u_{i-1}} = F_i \cup \{ {(e_i,u_{i-1})}\},
  \quad F_{u_i} = F_i \cup \{ {(e_i,u_i)}\} \]
for suitable $F_i$. Since $e_i$ and $e_{i+1}$ are successive red edges
we divide out by the variable difference $x_{e_i,u_i} - x_{e_{i+1},u_i}$
and so $(e_i,u_i)$ identifies with $(e_{i+1},u_i)$. We also have
\[ F_{u_{i}} = F_{i+1}  \cup \{ {(e_{i+1},u_i)}\},
  \quad F_{u_{i+1}} = F_{i+1} \cup \{ {(e_{i+1},u_{i+1})}\}. \]
We must then have
$F_i = F_{i+1}$. Hence all these $F_i$ are equal.

Suppose the edges on the path are not of the same color. Suppose going
from $v$ to $w$ there is first a sequence of red edges, the first one
being $e = \{ v = u_{0},u_1 \}$  and then eventually a blue edge
$f = \{u_i, u_{i+1}\}$. 
%
\begin{itemize}
\item The facet $F_v$ contains $(e,u_0)$ of color red. The facet $F_w$
also contains a (class) of a red edge. If this red edge was $e$ it would
have to be $(e,u_1)$. Hence $(e,u_0)$ is in $F_v$ but not in $F_w$.
\item Similarly the blue $(f,u_i)$ is in $F_v$, and by a similar argument
as above, $(f,u_i)$ is not in $F_w$. 
\ignore{
\begin{itemize}
\item The facet $F_v$ contains $(e,1)$ colored red.
  If $(e,1)$ is connected to $(g,*)$ by an edge in $\oPsi P$,
  then $g$ is a red edge
  on a different side of $v = u_{0}$ than $e$.
  It must be $(g,1)$ and with no other
  red edge on the path form $g$ to $v$. Furthermore if $(g,1)$ is
  connected by an edge to another vertex,
  it is either $(e,1)$ or a vertex $(g^\prime,1)$
  of the same type.

  In the facet $F_w$, $(e,1)$ does not occur,
  nor any of the $(g,1)$ above. So this class is not present in $\ovF_w$.
\item Consider the edge $(f,1)$. If it is connected to
  $(g,*)$ by an edge in $\oPsi \tau P$ it must again be $(g,1)$
  (the only way we could
  have $0$ as second coordinate would be some edge $(e_i,0)$ before
  we come to $f$, but $e_i$ is a red edge.)
  Also $g$ is on the opposite
  side of where $f$ is pointing. Similarly any edge connecting $(g,1)$
  must have the other vertex $(f,1)$ or of the same type as $(g,1)$.
  Thus the class of $(f,1)$ divides
  $\ovF_v$. However since $(f,1)$ nor the $(g,1)$'s do not occur in $F_w$,
  it does not divide $\ovF_w$.
}
\item The upshot is that $F_v \setminus F_w$ contains at least
  two elements, and so $F_v$ and $F_w$ do not intersect in
  codimension one.
\end{itemize}
\end{proof}

Recall that a set of edges in the tree $T$ is {\em independent}
if no two edges in the set are adjacent.
The quotient
of $\Spol/I(T)$ by $L(P)$ is a stacked simplicial complex.
It is again a quotient of the polynomial ring $\Spol/L(P)$.
Each part $E_i$ of $E$ is a subforest of $T$. Let $T_{ij}$ be the trees
of this subforest and $V_{ij} \sus V$ the support of $T_{ij}$.
Let $T^\prime$ be the hypertree whose edges are the sets $V_{ij}$.
In particular note that if $P$ is a partition whose parts $E_i$ consist of 
independent edges, then each $T_{ij}$ is simply an edge, and so $T^\prime = T$.

\begin{proposition}
  Let $P$ be a partition of the edge set of $T$.
The quotient
of $\Spol/I(T)$ by $L(P)$ corresponds to
a stacked simplicial complex $X$ whose associated
hypertree is $T^\prime$.
\end{proposition}

\begin{proof}

Consider then the tree $T_{ij}$. Let $v,u,w$ be three vertices in $V_{ij}$.
If they are $T$-aligned for some ordering, the facets $F_v, F_u, F_w$ of
$X$ have a
codimension-one face in common by Lemma \ref{lem:ball-cod}.
Suppose $\{u,v,w\}$ are non-aligned. Consider
the path from $v$ to $w$ and let $e$ be its last edge. Then $e$ is also
the last edge on the path from $u$ to $w$. Write
$F_w = F \cup \{(e,w)\}$. By the argument of Lemma \ref{lem:ball-cod},
all $x$ on these
paths have $F_x$ containing $F$. We readily get that $F$ is a 
codimension-one face of every $F_x$ for $x \in V_{ij}$.
Thus each $V_{ij}$ form an edge in the hypertree $T^\prime$
associated to the simplicial complex $X$. 
\end{proof}


\begin{theorem} \label{thm:ball-partedgeind}
There is a one-to-one correspondene between:
\begin{itemize}
\item regular linear spaces giving squarefree quotients of
$\Spol/I(T)$ corresponding to triangulated balls, and
\item 
  partitions $P$ of the edge set\/ $E$ of\/ $T$ into sets of independent edges.
  \end{itemize}
  The codimension-one faces of this triangulation which are on two facets are
  precisely the faces $F_e$ of Lemma~\ref{lem:ball-fe}.
  Let $B(T,P)$ be the ideal generated by the $x_{F_e}$ for $e \in E(T)$. Then
  $I(T) + B(T,P)$ is the 
  Stanley--Reisner ideal in $\Spol/L(P)$ definining the {\em boundary} of this
  triangulated ball, a triangulated
  sphere.
\end{theorem}

\begin{proof}
  When the edges are partitioned into independent sets, the hypertree
  $T^\prime$ is an ordinary tree $T$.
  And when a stacked simplicial complex gives an ordinary tree~$T$,
  it is a triangulated ball, and may be realized as a stacked polytope.

  The only faces on a stacked simplical complex not on the boundary, are
  the codimension one faces which are on at least two faces. This gives
  the statement about the Stanley--Reisner ideal of the boundary. 
  \end{proof}

\begin{remark} In \cite{DFN} the first author et al.~{give} the construction
  of large classes of triangulated balls, defined by letterplace ideals of
  posets. The ideal defining the boundary of triangulated balls is given
  in a similar way there.
\end{remark}


In particular
triangulations of simplicial polygons correspond to partitions of
trees $T$ into {\em three parts}, each part being a set of independent edges.
Thus only trees $T$ whose maximal vertex degree is $3$ arise
in this context.

\begin{corollary}
The length of the longest regular sequence of variable differences
giving a squarefree quotient of $\Spol/I(T)$ that corresponds to a triangulated
ball is $|E|- \Delta$, where $\Delta$ is the maximal degree of
a vertex of\/ $T$.
\end{corollary}

\begin{proof}
  This is because the minimal number of parts in a partition of $E(T)$ into
  independent edges, the edge chromatic number of the tree $T$,
  is the maximal degree of a vertex in $T$,
  \cite{Agn}.
\end{proof}

\appendix

\section{Partitions of the vertices and edges of a tree}
\label{sec:part}

We recall the basic result on trees from \cite{FlPar} on the correspondence
between partitions of edges and partitons of vertices into independent sets.
Let $T$ be a tree with vertex set $V$ and 
edge set $E$. We consider partitions of the vertices
\begin{equation} \label{eq:part-V}
  V = V_1 \sqcup V_2 \sqcup \cdots \sqcup V_r
\end{equation}
into disjoint sets such that each $V_i$ is an independent set of vertices.
(This is almost the same as a coloring of vertices, but not quite:
The symmetric group $S_r$ acts on colorings by permuting the color labels
of the $V_i$. So such a partition
is an orbit for the actions of $S_r$. The class of such orbits, or equivalently
of partitions \eqref{eq:part-V} are also called
{\it non-equivalent vertex colorings}, see \cite{HeMe}.)

We also consider partitions of the edges
\[ E = E_1 \sqcup E_2 \sqcup \cdots \sqcup E_s. \]
Here we have no independence requirements. Any partition is good.

\medskip
Now we make a correspondence as follows. Given such a partition of $V$,
make a partition of $E$ as follows: If $v$ and $w$ are vertices consider
the unique path in $T$ linking $v$ and $w$.
Let $f$, respectively
$g$, be the edge incident to $v$, respectively $w$, on this path.
If (i)~$v$ and $w$ are in the {\it same part}~$V_i$ of~$V$ and
(ii)~{\it no other} vertex
on this path is in the part $V_i$, then put $f$
and $g$ into the same part of~$E$, and write $f \sim^{\!\! \prime}_E g$. The partition
of edges is the equivalence relation $\sim_E$ {\it generated}
by $\sim^{\!\! \prime}_E$, i.e., the smallest equivalence relation on $E$ containing $\sim^{\!\! \prime}_E$. Note that in general $\sim^{\!\! \prime}_E$ alone would not be reflexive nor transitive.

\medskip Conversely, given a partition of the edge set $E$, make a partition
of $V$ as follows: Let $v$ and $w$ be distinct vertices, and consider
again the path from $v$ to $w$. If (i)~the edges $f$ and $g$ are distinct,
(ii)~$f$ and $g$ are in the {\it same
part} $E_j$, and (iii)~{\it no other} edge on this path is in the part $E_j$,
then put $v$ and $w$ in the same part of $V$, and write
$v \sim^{\!\! \prime}_V w$. The partition of vertices is the equivalence relation
$\sim_V$ generated by~$\sim^{\!\! \prime}_V$.

\begin{theorem}[\cite{FlPar}] \label{thm:part-VE}
Let $T$ be a tree with vertex set\/ $V$ and edge
set\/ $E$.
The above gives a one-to-one correspondence between partitions of the vertices
$V$ into $r+1$ {\it independent} sets, and partitions of the edges $E$
into $r$ sets.
\end{theorem}


\begin{example}
Any tree has a unique partition of the vertices into two independent sets
(two colors modulo $S_2$).
This corresponds to the partition of the edges into one part (one color).
\end{example}

\begin{example}\label{ex:redandblueeightedges}
  In Figure~\ref{figfive} 
  we partition the edges into red and black color classes. The vertices
  are then partitioned into three sets, each consisting of independent vertices. 
  The partition of the vertex set of the first tree is
  \[ \{1,3,5\}\cup\{2,6\}\cup\{4\},\]
  and that of the second tree is
  \[\{1,3,5,8,10\}\cup\{2,4,7\}\cup\{6,9\}.\]
\end{example}

\begin{figure}
\begin{center}
\begin{tikzpicture}[dot/.style={draw,fill,circle,inner sep=1pt},scale=1.2]
\draw [help lines,white] (1,-.5) grid (6,0);
\draw[very thick,red] (3,0)--(5,0);
\draw[very thick] (1,0)--(3,0);
\draw[very thick] (5,0)--(6,0);
\fill (1,0)  circle (0.09);
\fill (2,0)  circle (0.09);
\fill (3,0)  circle (0.09);
\fill (4,0)  circle (0.09);
\fill (5,0)  circle (0.09);
\fill (6,0)  circle (0.09);
\fill[yellow] (1,0)  circle (0.065);
\fill[green] (2,0)  circle (0.065);
\fill[yellow] (3,0)  circle (0.065);
\fill[blue] (4,0)  circle (0.065);
\fill[yellow] (5,0)  circle (0.065);
\fill[green] (6,0)  circle (0.065);
\node at (1,.25) {$1$};
\node at (2,.25) {$2$};
\node at (3,.25) {$3$};
\node at (4,.25) {$4$};
\node at (5,.25) {$5$};
\node at (6,.25) {$6$};
\end{tikzpicture}\\
\begin{tikzpicture}[dot/.style={draw,fill,circle,inner sep=1pt},scale=1.1]
\draw[very thick,red] (5,0)--(6,0);
\draw[very thick] (1,0)--(5,0);
\draw[very thick] (6,0)--(7,0);
\draw[very thick] (4,0)--(4,-1);
\draw[very thick,red] (4,-1)--(4,-3);
\fill (1,0)  circle (0.09);
\fill (2,0)  circle (0.09);
\fill (3,0)  circle (0.09);
\fill (4,0)  circle (0.09);
\fill (5,0)  circle (0.09);
\fill (6,0)  circle (0.09);
\fill (7,0)  circle (0.09);
\fill (4,-1)  circle (0.09);
\fill (4,-2)  circle (0.09);
\fill (4,-3)  circle (0.09);
\fill[yellow] (1,0)  circle (0.065);
\fill[green] (2,0)  circle (0.065);
\fill[yellow] (3,0)  circle (0.065);
\fill[green] (4,0)  circle (0.065);
\fill[yellow] (5,0)  circle (0.065);
\fill[blue] (6,0)  circle (0.065);
\fill[green] (7,0)  circle (0.065);
\fill[yellow] (4,-1)  circle (0.065);
\fill[blue] (4,-2)  circle (0.065);
\fill[yellow] (4,-3)  circle (0.065);
\node at (1,.25) {$1$};
\node at (2,.25) {$2$};
\node at (3,.25) {$3$};
\node at (4,.25) {$4$};
\node at (5,.25) {$5$};
\node at (6,.25) {$6$};
\node at (7,.25) {$7$};
\node at (4.22,-1) {$8$};
\node at (4.22,-2) {$9$};
\node at (4.3,-3) {$10$};
\end{tikzpicture}
\caption{}
\label{figfive}
\end{center}
\end{figure}


\bibliographystyle{amsplain}
\bibliography{biblio}

\end{document}